\documentclass[11pt,reqno,oneside]{amsart}

\usepackage{graphicx,subfigure,threeparttable}
\usepackage{amssymb, bm}
\usepackage{epstopdf}

\usepackage[a4paper, total={6in, 9in}]{geometry}

\usepackage{booktabs} 

\usepackage{amsfonts}

\usepackage{amsmath,amsthm,mathrsfs,amssymb}
\usepackage[usenames]{color}

\usepackage{enumerate}
\usepackage{hyperref}
\usepackage{xcolor}
\hypersetup{
  colorlinks,
  linkcolor={blue!80!black},
  urlcolor={blue!80!black},
  citecolor={blue!80!black}
}
\usepackage[capitalize,noabbrev]{cleveref}

\newtheorem{thm}{Theorem}[section]

\newtheorem{lem}{Lemma}[section]

\theoremstyle{definition}

\theoremstyle{remark}

\newtheorem{rem}{Remark}[section]
\numberwithin{equation}{section}

\numberwithin{equation}{section}

\newcounter{saveeqn}

\usepackage{refcheck}

\title[A novel electromagnetic imaging scheme]{Invisibility enables super-visibility in electromagnetic imaging}

\author{Youzi He}
\address{Department of Mathematics, The Chinese University of Hong Kong, Shatin, Hong Kong SAR, China}
\email{yolanda19he@gmail.com}

\author{Hongjie Li}
\address{Department of Mathematics, The Chinese University of Hong Kong, Shatin, Hong Kong SAR, China}
\email{hjli@math.cuhk.edu.hk}

\author{Hongyu Liu}
\address{Department of Mathematics, City University of Hong Kong, Kowloon, Hong Kong SAR, China}
\email{hongyu.liuip@gmail.com}

\author{Xianchao Wang}
\address{School of Mathematics, Harbin Institute of Technology, Harbin, China}
\email{xcwang90@gmail.com}

\date{} 
\begin{document}
\maketitle

\begin{abstract}

This paper is concerned with the inverse electromagnetic scattering problem for anisotropic media. We use the interior resonant modes to develop an inverse scattering scheme for imaging the scatterer. The whole procedure consists of three phases. First, we determine the interior Maxwell transmission eigenvalues of the scatterer from a family of far-field data by the mechanism of the linear sampling method. Next, we determine the corresponding transmission eigenfunctions by solving a constrained optimization problem. Finally, based on both global and local geometric properties of the transmission eigenfunctions, we design an imaging functional which can be used to determine the shape of the medium scatterer. We provide rigorous theoretical basis for our method. Numerical experiments verify the effectiveness, better accuracy and super-resolution results of the proposed scheme.

\medskip

\noindent{\bf Keywords:}~~ inverse electromagnetic scattering, anisotropic media, transmission eigenfunctions, geometric structures, super-resolution

\noindent{\bf 2010 Mathematics Subject Classification:}~~35Q60, 35P25, 35R30, 78A40

\end{abstract}

\section{Introduction}

\subsection{Mathematical setup}

We present the mathematical formulation of the forward and inverse electromagnetic scattering problems.

Let $D\subset \mathbb{R}^3$ be the support of an anisotropic inhomogeneity, where $D$ is a bounded, simply connected Lipschitz domain with a piecewise smooth boundary. Next, we introduce the medium configuration, which is characterized by two material parameters including the electric permittivity $\bm{\varepsilon}$ and the magnetic permeability $\bm{\mu}$. It is assumed that $\bm{\mu}(\bm{x})\equiv\mu_0 \bm{I}$, $\bm{x}\in\mathbb{R}^3$, where $\mu_0$ is a positive constant and $\bm{I}$ is the $3\times 3$ identity matrix.  Let the electric permittivity $\bm{\varepsilon} \in (L^{\infty}( \mathbb{R}^3))^{3\times 3} $ be a uniformly symmetric-positive-definite bounded matrix-valued function. It is assumed that $\bm{\varepsilon}(\bm{x})=\varepsilon_0\bm{I}$ for $\bm{x}\in\mathbb{R}^3\backslash\overline{D}$, where $\varepsilon_0$ is a positive constant. Define the wavenumber $k$ to be
\[
k^2 := \varepsilon_0 \mu_0 \omega^2,
\]
where $\omega\in\mathbb{R}_+$ signifies the temporal frequency of the electromagnetic waves. We define the matrix index of refraction of the anisotropic medium by
\begin{equation}\label{eq:ind_ref}
\bm{N}(\bm{x}):=\frac{\bm{\varepsilon}(\bm{x}) }{\varepsilon_0},\quad \bm{x}\in \overline{D},
\end{equation}
and assume that $(\bm{N}-\bm{I})^{-1}$ is bounded and satisfies
\begin{equation}\label{eq:mat_posi_def}
((\bm{N}-\bm{I})^{-1}\bm{\xi},\bm{\xi} )\geq \alpha |\bm{\xi}|^2 \quad \text{for any }\bm{\xi}\in\mathbb{C}^3 \text{  a.e.  in }D
\end{equation}
for some constant $\alpha >0$. Clearly, all the results of this paper also hold true in the particular case of isotropic media, i.e. $\bm N(\bm x)=n(x)\bm I$, where $n(x)$ is a positive piecewise smooth function in $\overline{D}$. We take the time-harmonic electromagnetic incident fields $\bm{E}^i$, $\bm{H}^i$ of the form:
\begin{equation}\label{eq:incEH}
\begin{split}
\bm{E}^i(\bm{x},\bm{d},\bm{q})&=\frac{\mathrm{i}}{k}\text{curl curl}\, \bm{q}\mathrm{e}^{\mathrm{i}k \bm{x}\cdot \bm{d}}
               =\mathrm{i}k( \bm{d}\times \bm{q})\times \bm{d} \mathrm{e}^{\mathrm{i}k \bm{x}\cdot \bm{d}},\\
\bm{H}^i( \bm{x}, \bm{d}, \bm{q})&=\text{curl}\, \bm{q}\mathrm{e}^{\mathrm{i}k \bm{x}\cdot \bm{d}}
               =\mathrm{i}k \bm{d}\times \bm{q}\mathrm{e}^{\mathrm{i}k \bm{x}\cdot \bm{d}},
\end{split}
\end{equation}
where $\mathrm{i}:=\sqrt{-1}$ is the imaginary unit, $\bm{d} \in \mathbb{S}^2:=\{x \in \mathbb{R}^3: |\bm{x}|=1 \}$ is the direction of propagation, and $\bm q \in \mathbb{R}^3$ is the polarization vector satisfying $\bm{q}\perp \bm{d}$. The wave scattering of $\bm{E}^i$, $\bm{H}^i$ by the medium scatterer $(D, \bm{N})$ leads to the following Maxwell equations for the interior electric and magnetic fields $\bm{E}, \bm{H}$, and the scattered electric and magnetic fields $\bm{E}^s, \bm{H}^s$:
\begin{equation}\label{eq:directPro}
\left\{
\begin{alignedat}{3}
&\text{curl}\, \bm{E}-\mathrm{i}k\bm{H}=0, \quad  &&\text{curl}\,\bm{H}+\mathrm{i}k \bm{N}(\bm{x}) \bm{E}=0 \quad &&\text{in }D, \medskip \\
&\text{curl}\, \bm{E}^s-\mathrm{i}k\bm{H}^s=0, \quad &&\text{curl}\, \bm{H}^s+\mathrm{i}k \bm{E}^s=0 \quad &&\text{in }\mathbb{R}^3\setminus \overline{D}, \medskip \\
&\bm{\nu}\times (\bm{E}^s+\bm{E}^i)-\bm{\nu}\times \bm{E}=0,\quad  &&\bm{\nu}\times (\bm{H}^s+\bm{H}^i)-\bm{\nu}\times \bm{H}=0\quad &&\text{on} \ \partial D, \medskip \\
& \lim\limits_{|\bm{x}|\to\infty} \left(\bm{H}^s\times \bm{x}-|\bm{x}|\bm{E}^s  \right)=0.
\end{alignedat}
\right.
\end{equation}
The last limit in \eqref{eq:directPro} is known as the Silver-M\"uller radiation condition, which signifies that the scattered fields are outgoing. The well-posedness of the scattering problem \eqref{eq:incEH}-\eqref{eq:directPro} is established in \cite{KM95} for the unique existence of a pair of solutions $\bm{E}, \bm{H}\in H_{loc}(\mathrm{curl}; \mathbb{R}^3)$. If we eliminate the magnetic field in \eqref{eq:directPro}, the scattering system is reduced into:
\begin{equation*}
\left\{
\begin{alignedat}{3}
&\text{curl\,curl}\, \bm{E}-k^2 \bm N(\bm x)\bm{E}=0,  \quad &&\text{in }D, \medskip \\
&\text{curl\,curl}\, \bm{E}^s-k^2\bm{E}^s=0, \quad &&\text{in }\mathbb{R}^3\setminus \overline{D}, \medskip \\
&\bm{\nu}\times (\bm{E}^s+\bm{E}^i)-\bm{\nu}\times \bm{E}=0,\quad &&\text{on} \ \partial D, \medskip \\
&\bm{\nu}\times (\text{curl}\, \bm{E}^s+\text{curl}\, \bm{E}^i)-\bm{\nu}\times \text{curl}\, \bm{E}=0,\quad &&\text{on} \ \partial D, \medskip \\
& \lim\limits_{|\bm{x}|\to\infty} \left(\text{curl}\, \bm{E}^s\times \bm{x}- \mathrm{i}k|\bm{x}|\bm{E}^s  \right)=0.
\end{alignedat}
\right.
\end{equation*}
with the unit outward normal $\bm{\nu}$ on $\mathbb{S}^2$.
The Silver-M\"uller radiation condition leads to the following asymptotic behavior of the scattered fields:
\begin{equation*}
\bm{Q}^s(\bm{x})=\frac{\mathrm{e}^{\mathrm{i}k|\bm{x}|}}{|\bm{x}|} \left\{ \bm{Q}_{\infty}(\hat{\bm{x}})+\mathcal{O}\left(\frac{1}{|\bm{x}|} \right) \right\},  \quad |\bm{x}|\to \infty, \quad \bm{Q}:=\bm{E}, \bm{H},
\end{equation*}
uniformly in all directions $\hat{\bm{x}}=\bm{x}/|\bm{x}|\in\mathbb{S}^2$. The vector fields $\bm{E}_{\infty}$ and $\bm{H}_{\infty}$ defined on the unit sphere $\mathbb{S}^2$ are known as the electric far-field pattern and magnetic far-field pattern, respectively. They satisfy
\begin{equation*}
\bm{H}_{\infty}=\bm{\nu}\times \bm{E}_{\infty} \quad \text{and} \quad \bm{\nu}\cdot \bm{E}_{\infty}=\bm{\nu}\cdot \bm{H}_{\infty}=0.
\end{equation*}
The inverse scattering problem that we are concerned with is to recover the medium scatterer $(D, \bm{N})$ by knowledge of the associated far-field data, which can be formulated as the following nonlinear operator equation:
\begin{equation}\label{eq:nonlin_ope_eqn}
\mathcal{F}(D,\bm{N})=\bm{E}_{\infty}(\hat{\bm{x}},\bm{d},\bm{q}; k),\quad   \hat{\bm{x}}\in\mathbb{S}^{2},\ \bm{d}\in\mathbb{S}^{2}, \ \bm{q}\in\mathbb{R}^{3},  \ k\in I:=[k_0,k_1],
\end{equation}
where $\mathcal{F}$ is defined by the Maxwell system \eqref{eq:incEH}-\eqref{eq:directPro}. We shall mainly consider recovering the shape of the scatterer, namely $D$, independent of its medium content $\bm{N}$.

\subsection{Overview of the proposed method and discussion}

The inverse scattering problem \eqref{eq:nonlin_ope_eqn} lies at the heart of many industrial applications including radar/sonar, medical imaging and nondestructive testing. Many methods have been developed for this practically important inverse problem and we refer to \cite{Ammari15, CC06, ColtonKress19} and the references cited therein for the existing developments in the literature. In principle, most of the methods make use of the ``visible patterns" for the reconstruction. Here, by visible patterns, we mean that the scattering data that one can observe in the far-field measurement. As a sharp contrast, we propose to reconstruct to scatterer by making use of ``invisible patterns", namely $\bm{E}_\infty=\bm{H}_\infty\equiv \bm{0}$.

If $\bm{E}_\infty=\bm{H}_\infty\equiv \bm{0}$, by Rellich's theorem \cite{ColtonKress19}, one has $\bm{E}^s=\bm{H}^s=\bm{0}$ in $\mathbb{R}^3\backslash\overline{D}$.
In such a case, one has by direct verifications that $\bm{E}|_{D}\in H_{loc}(\mathrm{curl}, D)$ and $\bm{E}_0=\bm{E}^i|_{D}\in H(\mathrm{curl}, D)$ fulfill the following PDE system:
\begin{equation}\label{eq:elec_TEV_Pro}
\left\{
\begin{array}{ll}
\text{curl\,curl}\,\bm{E}-k^2 \bm{N}(\bm{x})\bm{E} =0  & \text{in} \  D, \medskip \\
\text{curl\,curl}\,\bm{E}_0-k^2 \bm{E}_0 =0  & \text{in} \  D, \medskip \\
\bm{\nu} \times \bm{E} = \bm{\nu} \times \bm{E}_0  & \text{on} \  \partial D, \medskip \\
\bm{\nu} \times \text{curl}\, \bm{E} = \bm{\nu} \times \text{curl}\, \bm{E}_0 & \text{on} \  \partial D.
\end{array}
\right.
\end{equation}
The system \eqref{eq:elec_TEV_Pro} is referred to as the Maxwell interior transmission eigenvalue problem (cf. \cite{ColtonKress19}). That is, when invisibility occurs, the scattering patterns are trapped inside the scatterer $D$ (noting that outside the scatterer, $\bm{E}=\bm{E}^i$ and hence there is no scattering pattern to be observed). The proposed reconstruction scheme critically rely on the geometric properties of the transmission eigenfunctions to \eqref{eq:elec_TEV_Pro}. Roughly speaking, we have the following global and local geometric patterns of the transmission eigenfunction $\bm{E}_0$: globally, its $L^2$-energy is localized around $\partial D$; and locally, it is nearly vanishing around high-curvature point on $\partial D$, and especially around boundary point where $\bm{\nu}$ is discontinuous which can be regarded as having infinite curvature. These global and local geometric patterns of the Maxwell transmission eigenfunctions have received considerable studies recently in the literature. Since it is not the focus of the current study, we refer to \cite{BlastenLX,DLWW,DCL} for the theoretical and numerical justifications of those geometric patterns for the Maxwell transmission eigenfunctions, and we also refer to \cite{BlastenA, BlastenLLW, Blasten16, Blasten17, Blasten21, CDHLW, DDL, DJLZ, Liu20} for related studies on similar geometric patterns for the Helmholtz transmission eigenfunctions. Nevertheless, we would like to point out that all of the aforementioned geometric studies are concerned with transmission eigenfunctions associated with isotropic media. As an interesting byproduct of the current study, we numerically verify that the general geometric properties described above also hold for Maxwell transmission eigenfunctions associated with anisotropic media.

Based on the above discussion, the proposed method consists of three phases. First, we determine the interior Maxwell transmission eigenvalues of the unknown scatterer from the far-field data in \eqref{eq:nonlin_ope_eqn} by the mechanism of the linear sampling method. Next, we determine the approximations of the corresponding transmission eigenfunctions by solving a constrained optimization problem. Finally, based on both global and local geometric properties of the transmission eigenfunctions, we design an imaging functional which can be used to determine the shape of the medium scatterer. Our study follows a similar spirit in a recent paper \cite{CDHLW}, where we developed an imaging scheme for the inverse acoustic scattering problem governed by the Helmholtz system. In this paper, we extend our study to the more complicated and challenging electromagnetic scattering problem and moreover associated with anisotropic media. In order to provide a rigorous basis for the proposed method, we establish several technical results that are of independent interest to the scattering theory of electromagnetic waves. In sharp contrast, the counterparts of these technical results required in \cite{CDHLW} for the acoustic wave scattering were known in the literature.

We would like to mention that there are methods proposed for reconstructing $\mathbf{N}$ in \eqref{eq:nonlin_ope_eqn} by making use of the transmission eigenvalues $k^2$ associated with \eqref{eq:elec_TEV_Pro}. Nevertheless, in the proposed method, we make use of the quantitative geometric properties of the transmission eigenfunctions for the reconstruction. This enables us to attain several salient features, in particular the super-resolution effect of the reconstructed images. Here, we recall the Abbe diffraction which states that in modern optics, the resolution limit is roughly around half of the operating wavelength. In our numerical study, it is shown that if $D$ possesses a thin layer of high refractive index, i.e. $\bm{N}$ is relatively large around $\partial D$, then transmission eigenvalues to \eqref{eq:elec_TEV_Pro} can occur for even small $k$, and moreover, the geometric patterns mentioned earlier are more evident for the corresponding transmission eigenfunctions. That means, in such a case, even if the (unknown) scatterer is of a size much smaller than the operating wavelength $2\pi/k$ (with $k$ being a transmission eigenvalue), the proposed method can still effectively reconstruct the shape of the scatterer. Hence, in practical applications, if one can coat the scattering object with a high refractive-index material (through indirect means), super-resolution imaging can be achieved via the proposed method. This is the main reason to suggest the title of our paper that invisibility enables super-visibility in electromagnetic imaging. In addition, the super-resolution effect can also be observed in the local reconstruction of geometrically singular places of $\partial D$, say e.g. corners. Such a viewpoint has been advocated in \cite{CDHLW} for the inverse acoustic imaging and we corroborate it in this paper for the inverse electromagnetic imaging.

The rest of the paper is organized as follows. In section 2, we discuss the determination of the Maxwell transmission eigenvalues by knowledge of the far-field data. In section 3, we determine the corresponding transmission eigenfunctions. Finally, in section 4, an imaging functional is designed and several numerical examples are presented.

\section{Determination of transmission eigenvalues}\label{sec:2}

In this section, we determine the transmission eigenvalues from a knowledge of a family of the electric far-field patterns. Although this problem has been solved in previous literature \cite{CCH16}, for completeness and self-containedness, we briefly discuss the main process as well as the rationale behind the method.

For any given $\bm{z}\in\mathbb{R}^3$, let
\begin{equation*}
\bm{E}_{e,\infty}(\hat{\bm{x}},\bm{z}, \bm q; k)=\frac{\mathrm{i}k }{4\pi}(\hat{\bm{x}}\times \bm{q})\times \hat{\bm{x}}\mathrm{e}^{-\mathrm{i}k\hat{\bm{x}}\cdot \bm{z}}
\end{equation*}
be the far-field pattern of an electric dipole with source at $\bm{z}$ and polarization $\bm{q}$. The determination of the Maxwell transmission eigenvalues is based on the so-called linear sampling method (LSM) \cite{CCM11}, which is a qualitative method in inverse scattering theory. We now consider the following far-field equation
\begin{equation}\label{eq:FFEqn}
(F_k \bm{g})(\hat{\bm{x}})=\bm{E}_{e,\infty}(\hat{\bm{x}},\bm{z},\bm{q};k), \quad \bm{z},\bm{q}\in \mathbb{R}^3,\ k\in I,
\end{equation}
where the far-field operator $F_k:L_t^2(\mathbb{S}^2)\to L_t^2(\mathbb{S}^2)$ is defined by
\begin{equation}\label{eq:FFO}
(F_k \bm{g})(\hat{\bm{x}}):=\int_{\mathbb{S}^2}\bm{E}_{\infty}(\hat{\bm{x}},\bm{d},\bm{g}(\bm{d});k)\, \text{d}s(\bm{d}),
\end{equation}
with $\bm{E}_{\infty}(\hat{\bm{x}},\bm{d},\bm{g}(\bm{d});k)$ the far-field pattern corresponding to the scattering problem \eqref{eq:incEH}-\eqref{eq:directPro}.
Since $\bm{g}(\bm{d})$ is implicit in the far-field pattern $\bm{E}_{\infty}(\hat{\bm{x}},\bm{d},\bm{g}(\bm{d});k)$, we can not solve the equation \eqref{eq:FFEqn} directly.

Next, we present the numerical method to identify $\bm{g}(\bm{d})$ in  \eqref{eq:FFEqn}.
 For more details about the numerical calculation, we refer the reader to \cite{CCM11}.
Let $\bm p\in \mathbb{R}^3$ be a unit auxiliary vector such that $\bm p\times \bm d\neq 0$ for any $\bm d\in \mathbb{S}^2$. Noting that $\bm d\cdot \bm g(\bm d)=0$, then the Herglotz kernel for each incident direction $\bm d$ can  be expanded as
\begin{equation*}
  \bm g(\bm d)=g^{\theta}(\bm d)\, {\bm p}^{\theta}+g^{\phi}(\bm d)\, {\bm p}^{\phi},
\end{equation*}
where $g^{\theta}, g^{\phi}\in \mathbb{C}$ are undetermined coefficients and
\begin{equation*}
  {\bm p}^{\theta}=\frac{\bm p\times \bm d}{|\bm p\times \bm d|}, \quad
  {\bm p}^{\phi}=\frac{(\bm p\times \bm d)\times \bm d }{|(\bm p\times \bm d)\times \bm d|}
\end{equation*}
are two mutually orthogonal polarizations. Thus, the electric far-field pattern $\bm{E}_{\infty}$
can be represented by
\begin{equation*}
  \bm{E}_{\infty}(\hat{\bm{x}},\bm{d},\bm{g} (\bm{d});k)=
  g^{\theta} \bm{E}_{\infty}(\hat{\bm{x}},\bm{d}, {\bm p}^{\theta};k)
  + g^{\phi} \bm{E}_{\infty}(\hat{\bm{x}},\bm{d}, {\bm p}^{\phi};k).
\end{equation*}
Therefore, the far-field operator $F_k$ can be rewritten as
\begin{equation*}
  (F_k \bm{g})(\hat{\bm{x}}):= \int_{\mathbb{S}^2}\left(g^{\theta}(\bm d;\bm z, \bm q)\, \bm{E}_{\infty}(\hat{\bm{x}},\bm{d}, {\bm p}^{\theta};k)
   + g^{\phi}(\bm d; \bm z, \bm q)\,  \bm{E}_{\infty}(\hat{\bm{x}},\bm{d}, {\bm p}^{\phi}; k) \right)\, \mathrm{d}s(\bm d).
\end{equation*}
As the far-field equation \eqref{eq:FFEqn} is ill posed, we usually adopt Tikhonov regularization \cite{ColtonKress19} and instead solve
\begin{equation}\label{eq:Tik_regu}
(\delta \bm{I}+F_k^{\ast}F_k)\bm{g}_{\delta}=F_k^{\ast}\bm{E}_{e,\infty}(\hat{\bm{x}},\bm{z},\bm{q};k),
\end{equation}
where $\delta>0$ is the regularization parameter and $F_k^{\ast}$ is the adjoint to $F_k$.

Now we introduce the rigorous mathematical justification of the LSM for determining the transmission eigenvalues. Define the Herglotz wave functions by
\begin{equation}\label{eq:ElecHerzPair}
\bm E_{\bm g,k}(\bm x):= \int_{\mathbb{S}^2} \bm g(\bm d)\mathrm{e}^{\mathrm{i}k \bm x\cdot \bm d}\text{d}s(\bm d), \quad \bm H_{\bm g,k}(\bm x):=\frac{1}{\mathrm{i}k}\text{curl}\, \bm E_{\bm g,k}(\bm x),
\end{equation}
where $\bm g\in L_t^2(\mathbb{S}^2)$ is called the vector Herglotz kernel of the pair $(\bm E_{\bm g,k}, \bm H_{\bm g,k})$.
The existence and discreteness of infinitely real transmission eigenvalues was established in \cite{CCH16}. Moreover, an effective discriminant method is proposed to distinguish whether $k$ is a Maxwell transmission eigenvalue or not in the following lemma\cite{CCH16}.

\begin{lem}\label{lem:tev}
Let $\bm{g}_{\delta}(\cdot, \bm z)\in L_t^2(\mathbb{S}^2)$ be the Tikhonov regularized solution of the far-field equation, i.e., the solution of \eqref{eq:Tik_regu}. Let $\bm{E}_{\bm{g}_{\delta},k}$ be the electric Herglotz wave function with kernel $\bm{g}_{\delta}$ defined in \eqref{eq:ElecHerzPair}. Then, for any ball $B \subset D$, $\| \bm{E}_{\bm{g}_{\delta},k} \|_{(L^2(D))^3}$ is bounded as $\delta \to 0$ for a.e. $\bm{z}\in B$ if and only if $k$ is not a transmission eigenvalue.
\end{lem}

\begin{rem}
There exists a constant $M>0$ such that $\| \bm{E}_{\bm{g}_{\delta},k} \|_{(L^2(D))^3}\leq M \| \bm{g}_{\delta}(\cdot, \bm z)\|_{(L^2(\mathbb{S}^2))^3}$. Then, according to Lemma \ref{lem:tev}, we note that $\| \bm{g}_{\delta}(\cdot, \bm z)\|_{(L^2(\mathbb{S}^2))^3}$ behaves quite differently whether $k$ is a Maxwell transmission eigenvalue or not. Although $\| \bm{E}_{\bm{g}_{\delta},k} \|_{(L^2(D))^3}$ also behaves differently when $k$ is a transmission eigenvalue or not, as $D$ is unknown, we therefore use $\| \bm{g}_{\delta}(\cdot, \bm z)\|_{(L^2(\mathbb{S}^2))^3}$ as an indicator to determine whether or not $k$ is a Maxwell transmission eigenvalue.
\end{rem}

As discussed above, we formulate the following scheme and summarize the procedure of determining transmission eigenvalues.
\begin{table}[h]
\begin{tabular}{p{1.5cm} p{11cm}}
\toprule
\multicolumn{2}{l}{\textbf{Algorithm \uppercase\expandafter{\romannumeral1}:}\ Transmission eigenvalues determination scheme}\\
\midrule
Step 1 & Find a unit auxiliary vector $\bm p$ and collect a pair of electric far-field data $\bm{E}_{\infty}(\hat{\bm{x}},\bm{d}, {\bm p}^{\theta};k)$ and $\bm{E}_{\infty}(\hat{\bm{x}},\bm{d}, {\bm p}^{\phi};k)$ for $(\hat{\bm{x}},\bm{d},k)\in \mathbb{S}^2\times \mathbb{S}^2   \times I$, where $I$ is a finite interval in $\mathbb{R}_{+}$.\\
Step 2 & Pick a point $\bm{z} \in D$ (a priori information), and take three independent artificial polarization vectors $\bm{q}_1=[1,0,0]^{\top},\bm{q}_2=[0,1,0]^{\top}$ and $\bm{q}_3=[0,0,1]^{\top}$. Then, for each $k\in I$, solve \eqref{eq:Tik_regu} to obtain the solutions $\bm{g}_{\delta}(\cdot, \bm{z},\bm{q}_1;k)$, $\bm{g}_{\delta}(\cdot,\bm{z},\bm{q}_2;k)$ and $\bm{g}_{\delta}(\cdot, \bm{z},\bm{q}_3;k)$.\\ 
Step 3 & Plot
$\sum_{\ell=1}^3\|\bm{g}_{\delta}(\cdot,\bm{z},\bm{q}_{\ell};k) \|_{(L^2(\mathbb{S}^2))^3}^2$
 against $k\in I$ and find the transmission eigenvalues where peaks appear in the graph.\\
\bottomrule
\end{tabular}
\end{table}


\section{Determination of transmission eigenfunctions}\label{sec:3}

 In section \ref{sec:2}, we have determined a number of Maxwell transmission eigenvalues within the interval $I$ by knowledge of  a family of the electric far-field patterns in \eqref{eq:nonlin_ope_eqn}.
 Next, we are devoted to determining the corresponding transmission eigenfunctions.
 To that end, we first show the following denseness result of the Herglotz wave in a Sobolev space.

\begin{lem}\label{le:ElcHergPrDens}(see \cite{CCM11, ColtonKress01}).
The set of Herglotz wave functions $\bm{E}_{\bm{g},k}$ defined by \eqref{eq:ElecHerzPair} is dense in $\overline{M(D)}$ with respect to the $H(\mathrm{curl},D)$ norm, where
\[
M(D):=\left\{\bm{E}\in C^2(D)\cap C^1(\overline{D}):\text{curl\,curl}\,\bm{E}=k^2\bm{E} \ \ \mathrm{ in }\ D \right\}.
\]
\end{lem}

The following theorem states that if $k\in\mathbb{R}_+$ is a Maxwell transmission eigenvalue, then there exists a Herglotz wave function $\bm{E}_{\bm{g}_{\epsilon},k}$ such that the scattered field corresponding to this $\bm{E}_{\bm{g}_{\epsilon},k}$ as the incident field is nearly vanishing. 

\begin{thm}\label{th:EigenfE}
Suppose that $k\in\mathbb{R}_+$ is a Maxwell transmission eigenvalue in $D$. For any sufficiently small $\epsilon \in \mathbb{R}_{+}$, there exists $\bm{g}_{\epsilon}\in L_t^2(\mathbb{S}^2)$ such that
\[
\|F_k \bm{g}_{\epsilon} \|_{(L^2(\mathbb{S}^2))^3}=\mathcal{O}(\epsilon) \ \ \text{and} \ \   \|\bm{E}_{\bm{g}_{\epsilon},k} \|_{H(\mathrm{curl},D)}=1+\mathcal{O}(\epsilon),
\]
where $F_k$ is the far-field operator defined by \eqref{eq:FFO} and $\bm{E}_{\bm{g}_{\epsilon},k}$ is the Herglotz wave function defined by \eqref{eq:ElecHerzPair} with the kernel $\bm{g}_{\epsilon}$.
\end{thm}

\begin{proof}
Let $\bm{E}_{0,k}$ be a normalized Maxwell transmission eigenfunction in $D$ associated with the Maxwell transmission eigenvalue $k$, which means that $\bm{E}_{0,k}$ with $\|\bm{E}_{0,k}\|_{H(\mathrm{curl},D)}=1$ is a solution of
\[
\text{curl curl}\,\bm{E}_{0,k}-k^2\bm{E}_{0,k}  =0 \quad \text{in } D.
\]
By Lemma \ref{le:ElcHergPrDens}, for any sufficiently small $\epsilon >0$, there exists $\bm{g}_{\epsilon}\in L_t^2(\mathbb{S}^{2})$ such that
\begin{equation*}
\|\bm{E}_{\bm{g}_{\epsilon},k} -\bm{E}_{0,k} \|_{H(\mathrm{curl},D)}< \epsilon,
\end{equation*}
where $\bm{E}_{\bm{g}_{\epsilon},k}$ is the Herglotz wave function with the kernel $\bm{g}_{\epsilon}$. Then, by the triangle inequality,
\[
\|\bm{E}_{\bm{g}_{\epsilon},k} \|_{H(\mathrm{curl},D)}\leq \|\bm{E}_{\bm{g}_{\epsilon},k} -\bm{E}_{0,k} \|_{H(\mathrm{curl},D)}+\| \bm{E}_{0,k} \|_{H(\mathrm{curl},D)}< \epsilon +\| \bm{E}_{0,k} \|_{H(\mathrm{curl},D)},
\]
and
\[
\|\bm{E}_{\bm{g}_{\epsilon},k} \|_{H(\mathrm{curl},D)}\geq \| \bm{E}_{0,k} \|_{H(\mathrm{curl},D)}-\|\bm{E}_{\bm{g}_{\epsilon},k} -\bm{E}_{0,k} \|_{H(\mathrm{curl},D)} > \| \bm{E}_{0,k} \|_{H(\mathrm{curl},D)}- \epsilon.
\]
So, one has that $\|\bm{E}_{\bm{g}_{\epsilon},k} \|_{H(\mathrm{curl},D)} =1+\mathcal{O}(\epsilon)$.

Besides, from the definition of the far-field operator, $F_k \bm{g}_{\epsilon}$ is the far-field pattern produced by the incident wave $\bm{E}_{\bm{g}_{\epsilon},k}$. According to Theorem 2.5 in \cite{LWZ}, one obtains that
\[
\|F_k \bm{g}_{\epsilon} \|_{(L^2(\mathbb{S}^{2}))^3}<C\epsilon,
\]
where $C = C(D,k)$ is a positive constant.

The proof is complete.
\end{proof}

On the basis of Theorem \ref{th:EigenfE} and normalization if necessary, we can say that the following optimization problem:
\[
\min\limits_{\bm{g}\in L_t^2(\mathbb{S}^2)} \|F_k \bm{g} \|_{(L^2(\mathbb{S}^2))^3} \quad \text{    s.t.  } \|\bm{E}_{\bm{g},k} \|_{H(\mathrm{curl},D)}=1
\]
exists at least one solution $\bm{g}_0 \in L_t^2(\mathbb{S}^2)$ when $k\in \mathbb{R}_+$ is a Maxwell transmission eigenvalue in $D$. Unfortunately, since $D$ is unknown, the constraint $\|\bm{E}_{\bm{g},k} \|_{H(\mathrm{curl},D)}=1$ in the problem above is unpractical. However, it is reasonable to address this issue by considering an alternative optimization problem:
\begin{equation}\label{eq:AlterOptnEle}
\min\limits_{\bm{g}\in L_t^2(\mathbb{S}^2)} \|F_k \bm{g} \|_{(L^2(\mathbb{S}^2))^3} \quad \text{    s.t.  } \|\bm{E}_{\bm{g},k} \|_{H(\mathrm{curl},\Omega)}=1,
\end{equation}
where $\Omega$ is an a priori ball containing $D$.

Let $\bm{g}_0$ be a ``satisfactory'' solution to the optimization problem \eqref{eq:AlterOptnEle}. Next, we illustrate that the Herglotz wave $\bm{E}_{\bm{g}_0,k}$ with the Herglotz kernel $\bm{g}_0$ is indeed an approximation to the Maxwell transmission eigenfunction $\bm{E}_{0,k}$ associated with the transmission eigenvalue $k$.
Before further discussion, we give a quantitative Rellich's result for the electromagnetic medium scattering, which is the generalization of the argument in the acoustic scattering case \cite{Blasten16}.

\begin{lem}\label{le:Rell}
Let $R>1, \rho>0$, and $\bm{E}^i, \bm{H}^i \in W^{1,3+\rho}_{loc}(\mathrm{div};\mathbb{R}^3)\cap H_{loc}(\mathrm{curl};\mathbb{R}^3)$ be the incident fields, satisfying
\[
\text{curl}\,\bm{E}^i-\mathrm{i}k\bm{H}^i=0, \ \ \text{curl}\,\bm{H}^i+\mathrm{i}k\bm{E}^i=0,
\]
with $\| \bm{E}^i \|_{ W^{1,3+\rho}(\mathrm{div}; B_{2R})} \leq \mathcal{I}$. Besides, let $\bm{E}, \bm{H}\in W^{1,3+\rho}_{loc}(\mathrm{div};\mathbb{R}^3)\cap H_{loc}(\mathrm{curl};\mathbb{R}^3)$ be the total electromagnetic fields satisfying
\[
\text{curl}\,\bm{E}-\mathrm{i}k\bm{H}=0, \ \ \text{curl}\,\bm{H}+\mathrm{i}k\bm{N}(\bm{x}) \bm{E}=0,
\]
where the matrix index of refraction $\bm{N}(\bm{x})$ is defined in \eqref{eq:ind_ref}. The scattered electric and magnetic fields $\bm{E}^s=\bm{E}-\bm{E}^i , \bm{H}^s=\bm{H}-\bm{H}^i$ satisfy the Silver–M\"uller radiation condition. Let $\bm{E}_{\infty}, \bm{H}_{\infty}$ be their far-field patterns. We further assume that $\| \bm{E}^s \|_{(W^{1,3+\rho}(B_{2R}))^3}\leq \mathcal{T}$.
There is $\epsilon_0>0$ such that if
\[
\|\bm{E}_{\infty} \|_{(L^2(\mathbb{S}^2))^3} \leq \epsilon_0,
\]
then for some positive constant $\beta \leq \frac{\rho}{3+\rho}$ depending on $R, \bm{N}$, there is
\[
\sup_{\partial D}|\bm{E}^s|\leq \mathcal{C}\left(\mathrm{ln\,ln}(\mathcal{T} \|\bm{E}_{\infty} \|_{(L^2(\mathbb{S}^2))^3}^{-1} ) \right)^{-\beta}
\]
for some $\mathcal{C}=C(\bm{N},\omega ,R)(\mathcal{T}+\mathcal{I})$.
\end{lem}

\begin{proof}
In what follows, we assume that $D\subset B_R$, where $B_R$ is a central ball of radius $R\in \mathbb{R}_+$.
Clearly, the scattered electric field $\bm{E}^s$ satisfies the vector Helmholtz equation outside $D$. That is to say, the Cartesian components of $\bm{E}^s$ satisfy the scalar Helmholtz equation outside $D$. According to Sobolev embedding and the far-field to near-field estimate for the acoustic scattering (see e.g. Corollary 5.3 in \cite{Blasten16}), there are constants $C,c>0$ depending on $k,R,B$ such that
\begin{equation}\label{eq:est_Es_B}
\|E^s_j \|_{L^{\infty}(B)}\leq \|E^s_j \|_{W^{1,3+\rho}(B)}\leq C\mbox{ max}\left(\|E_{\infty,j}\|_{L^2(\mathbb{S}^2)} , \mathcal{T}\mathrm{e}^{-c\sqrt{\mbox{ln}\left(\mathcal{T}/ \|E_{\infty,j}\|_{L^2(\mathbb{S}^2)}\right) }}\right),
\end{equation}
where the elements with subscripts $j=1,2,3$ stand for the Cartesian components of the corresponding vectors, and $B$ is a domain such that $\overline{B}\subset B_{2R}\setminus \overline{B}_R$.
One of our requirements on $\epsilon_0$ is that the maximum picks the number on the right side of \eqref{eq:est_Es_B}. Therefore we must have $\epsilon_0 \leq \mathcal{T} \mathrm{e}^{-c^2}$.

Besides, we assume that $\| \bm{E}^s \|_{(L^{\infty}( B_{(2-\lambda)R}\setminus B_{(1+\lambda)R}))^3}\leq \gamma \leq 1$. Then according to Proposition 5.7 in \cite{Blasten16},
\begin{equation}\label{eq:est_Es_57}
|\bm{E}^s(\bm{x})|\leq C\mathcal{T}\gamma^{c_2^{(2+\lambda)R/r+2}}
\end{equation}
for $\bm{x}\in B_{2R}\setminus B(D,4r)$, and some positive constants $C, c_2<1/4, 1/4<\lambda<1/2, r<(1-2\lambda)R/2$.
Here, $B(D,4r)=\cup_{\bm{x}\in D}B(\bm{x},4r) $.
Then, when we take $\bm{x}'\in B_{3R/2}\setminus B(D,4r)$, by the H\"older continuity of $\bm{E}^s$ (see e.g. Theorem 4 in \cite{Alberti}) and \eqref{eq:est_Es_57}, there exists $0<\beta \leq \frac{\rho}{3+\rho}$ depending on $R, \bm{N}$ such that
\begin{equation}\label{eq:est_Es_modu}
\begin{split}
|\bm{E}^s(\bm{x})|\leq &\|\bm{E}^s\|_{C^{0,\beta}(\overline{B}_{3R/2} )}|\bm{x}-\bm{x}'|^{\beta}+|\bm{E}^s(\bm{x}')|\\
\leq &C8^{\beta}r^{\beta}\left( \|\bm{E}^s \|_{(L^2(B_{3R/2}))^3}+ \|\bm{E}^s \|_{(W^{1,3+\rho}(B_{3R/2}))^3}+\| \bm{E}^i \|_{W^{1,3+\rho}(\mathrm{div};B_{3R/2})} \right)\\
&+C\mathcal{T}\gamma^{c_2^{(2+\lambda)R/r+2}}\\
\leq &C(\mathcal{T}+\mathcal{I} )8^{\beta}r^{\beta}+C\mathcal{T}\gamma^{c_2^{(2+\lambda)R/r+2}}
\end{split}
\end{equation}
for $d(\bm{x},\partial D)\leq 4r$. The constant $C$ depends on $R, \bm{N},  \omega$.

Next, we require
\[
\gamma= C\mathcal{T} \mathrm{e}^{-c\sqrt{\mbox{ln}(\mathcal{T}/\epsilon_0 ) }}<1/\text{exp exp}\,(C(\lambda,k,R) )
\]
and
\[
r=r\left((\mbox{ln}|\mbox{ln}\gamma |)^{-1} ,\lambda,k,R \right).
\]
If we further choose $\lambda=\lambda(R)$ and let $\epsilon_0$ small enough, the claim follows from \eqref{eq:est_Es_modu}.

\end{proof}

\begin{rem}
Lemma ~\ref{le:Rell} shows that the smallness of the far-field pattern can ensure that the scattered field on the boundary are also small.
\end{rem}

\begin{rem}
We note that Lemma ~\ref{le:Rell} is analyzed in terms of electric fields. If we do the same analysis for the magnetic fields, we can get similar results, that is $| \bm{H}^s | \leq \psi (\epsilon_0) $ on $\partial D$ if $\|\bm{E}_{\infty} \|_{(L^2(\mathbb{S}^2))^3} \leq \epsilon_0$. Here, $\psi$ is a double logarithmic type function and satisfies $\psi (\epsilon_0)\to 0$ as $\epsilon_0\to +0$.
\end{rem}

\begin{rem}
We would like to emphasize that the quantitative Rellich's result for the electromagnetic case is verified here for the first time. Actually, there is a higher requirement for the regularity of the scatterer \cite{Blasten16}. However, we are not overly concerned about this issue, which is not the focus of this paper. Nevertheless, we believe that this regularity assumption can be relaxed. In fact, the subsequent numerical examples demonstrate the effectiveness of our method even if $\partial D$ is only Lipschitz continuous.
\end{rem}

In the following theorem, we demonstrate that the Herglotz wave $\bm{E}_{\bm{g}_0,k}$ is indeed generically an approximation to the Maxwell transmission eigenfunction $\bm{E}_{0,k}$.

\begin{thm}\label{th:IndeedAppro}
Suppose that the assumptions in Lemma ~\ref{le:Rell} hold. Additionally, we suppose that $k\in\mathbb{R}_+$ is a Maxwell transmission eigenvalue in $D$ and $\bm{g}_0$ is a solution to the optimization problem \eqref{eq:AlterOptnEle} satisfying
\begin{equation}\label{eq:ssn1}
\|F_k \bm{g}_0\|_{(L^2(\mathbb{S}^{2}))^3}\leq \epsilon \ll 1.
\end{equation}
Then the Herglotz wave $\bm{E}_{\bm{g}_0,k}$ is an approximation to a transmission eigenfunction $\bm{E}_{0,k}$ associated with the transmission eigenvalue $k$ in $(L^2(D))^3$-norm.
\end{thm}

\begin{proof}
Consider the scattering system \eqref{eq:incEH}-\eqref{eq:directPro}. Let $\bm{E}^i=\bm{E}_{\bm{g}_0,k}$, $\bm{E}^s$ and $\bm{E}^t$ be, respectively, the incident, scattered and total electric fields. After expressing the magnetic fields in terms of the electric fields, one has
\begin{equation}\label{eq:EleMedSys}
\left\{
\begin{array}{ll}
\text{curl\,curl}\,\bm{E}^t-k^2 \bm{N}(\bm{x})\bm{E}^t =0 \quad &\text{in} \ D, \medskip \\
\text{curl\,curl}\,\bm{E}_{\bm{g}_0,k}-k^2 \bm{E}_{\bm{g}_0,k}=0\quad &\text{in} \ D, \medskip \\
\bm{\nu}\times (\bm{E}_{\bm{g}_0,k}+\bm{E}^s)=\bm{\nu}\times \bm{E}^t  \quad &\text{on} \ \partial D, \medskip \\
\bm{\nu}\times (\text{curl}\, \bm{E}_{\bm{g}_0,k}+\text{curl}\,\bm{E}^s)=\bm{\nu}\times \text{curl}\,\bm{E}^t \quad &\text{on} \ \partial D.
\end{array}
\right.
\end{equation}
According to our earlier discussion, $F_k\bm{g}_0$ is the far-field pattern of $\bm{E}^s$. By virtue of Lemma ~\ref{le:Rell} and \eqref{eq:ssn1}, there is
\begin{equation}\label{eq:Esmall}
| \bm{\nu} \times \bm{E}^s | \leq \psi (\epsilon) \text{ and }  | \bm{\nu} \times \text{curl}\,\bm{E}^s | \leq \psi (\epsilon)  \quad \text{on }\partial D,
\end{equation}
where $\psi$ is the stability function, which is of double logarithmic type and satisfies $\psi (\epsilon)\to 0$ as $\epsilon\to +0$. For more details about the quantitative Rellich's theorem and the stability function $\psi$, we recommend that readers refer to \cite{Blasten16}.

Recall the Maxwell interior transmission eigenvalue problem \eqref{eq:elec_TEV_Pro}. Setting
\begin{equation}\label{eq:u_v}
\bm{u}=\bm{E}-\bm{E}_0 \quad \text{and} \quad \bm{v}=\bm{N}(\bm{x})\bm{E}-\bm{E}_0,
\end{equation}
one has the following fourth order PDE of $\bm{u}$:
\begin{equation*}\label{eq:u_4th}
(\text{curl curl} -k^2\bm{N} )(\bm{N}-\bm{I})^{-1}(\text{curl curl}\,\bm{u}-k^2\bm{u} )=0 \quad \text{in} \ D.
\end{equation*}
After using a variational framework, the system \eqref{eq:elec_TEV_Pro} becomes
\begin{equation}\label{eq:VarFrame}
\mathcal{A}_k (\bm{u},\bm{u}')-k^2 \mathcal{B}(\bm{u}, \bm{u}')=0, \text{  for any } \bm{u}' \in H_0^2(\mathrm{curl}, D),
\end{equation}
where $\mathcal{A}_k$ and $\mathcal{B}$ are continuous sesquilinear forms on $H^2(\mathrm{curl}, D)\times H^2(\mathrm{curl}, D)$, defined by
\begin{equation*}\label{eq:ses_Ak}
\mathcal{A}_k (\bm{u}, \bm{u}')=\left((\bm{N}-\bm{I})^{-1}(\text{curl curl}\, \bm{u}-k^2\bm{u}), \text{curl curl}\, \bm{u}'-k^2 \bm{u}' \right)_D+k^4(\bm{u},\bm{u}')_D,
\end{equation*}
and
\begin{equation*}\label{eq:ses_B}
\mathcal{B}(\bm{u},\bm{u}')=(\text{curl}\, \bm{u}, \text{curl}\, \bm{u}')_D.
\end{equation*}
With the help of the Riesz representation theorem, we define two bounded linear operators $A_k:H_0^2(\mathrm{curl}, D)\to H_0^2(\mathrm{curl}, D)$ and $B:H_0^2(\mathrm{curl}, D)\to H_0^2(\mathrm{curl}, D)$ by
\begin{equation*}\label{eq:op_Ak}
(A_k \bm{u}, \bm{u}')_{H^2(\mathrm{curl}, D)}:=\mathcal{A}_k (\bm{u},\bm{u}') \quad \text{and}\quad (B\bm{u},\bm{u}')_{H^2(\mathrm{curl}, D)}:=\mathcal{B} (\bm{u},\bm{u}').
\end{equation*}


Next, we consider the scattering system \eqref{eq:EleMedSys}, by \eqref{eq:Esmall} and some analytical regularization techniques (i.e., mollifiers), we could take $\bm{\zeta} \in H^2(\mathrm{curl}, D)$ such that
\begin{equation}\label{eq:zeta_boundary}
\bm{\nu} \times \bm{\zeta}=\bm{\nu} \times \bm{E}^s, \quad \bm{\nu} \times \text{curl}\, \bm{\zeta} = \bm{\nu} \times \text{curl}\, \bm{E}^s \quad \text{on }\partial D,
\end{equation}
and
\begin{equation}\label{eq:zeta_small}
\| \bm{\zeta} \|_{H^2(\mathrm{curl}, D)}\leq C\psi(\epsilon),
\end{equation}
where $C$ is a constant depending on $D$. Setting $\widetilde{\bm{E}}:=\bm{E}^t-\bm{\zeta}\in H^2(\mathrm{curl}, D)$, the system \eqref{eq:EleMedSys} could be rewritten as
\begin{equation}\label{eq:EleMedSys2}
\left\{
\begin{array}{ll}
\text{curl curl}\,\widetilde{\bm{E}}-k^2 \bm{N}(\bm{x})\widetilde{\bm{E}} =-(\text{curl curl} -k^2 \bm{N}(\bm{x}))\bm{\zeta} \quad &\text{in} \ D, \medskip \\
\text{curl curl}\,\bm{E}_{\bm{g}_0,k}-k^2 \bm{E}_{\bm{g}_0,k}=0\quad &\text{in} \ D, \medskip \\
\bm{\nu}\times \widetilde{\bm{E}}= \bm{\nu}\times \bm{E}_{\bm{g}_0,k}  \quad &\text{on} \ \partial D, \medskip \\
\bm{\nu}\times \text{curl}\,\widetilde{\bm{E}} =\bm{\nu}\times \text{curl} \,\bm{E}_{\bm{g}_0,k} \quad &\text{on} \ \partial D.
\end{array}
\right.
\end{equation}
Then, we take $\widetilde{\bm{u}}:=\widetilde{\bm{E}}- \bm{E}_{\bm{g}_0,k}\in H_0^2(\mathrm{curl}, D)$. According to the variational formulae, the PDE system \eqref{eq:EleMedSys2} can be rewritten as
\begin{equation}\label{eq:EleMedSys2VF}
(A_k\widetilde{\bm{u}}, \bm{u}' )-k^2(B \widetilde{\bm{u}}, \bm{u}' )=\mathfrak{f}  (\bm{u}'),\quad \text{for any } \bm{u}' \in H_0^2(\mathrm{curl}, D),
\end{equation}
where $\mathfrak{f}(\bm{u}')$ is an antilinear form on $H^2(\mathrm{curl}, D)$, defined by
\begin{equation}\label{eq:anti_f}
\mathfrak{f}(\bm{u}')=(-(\bm{N}-\bm{I})^{-1}(\text{curl curl} -k^2 \bm{N}(\bm{x}))\bm{\zeta},(\text{curl curl}-k^2) \bm{u}')_D.
\end{equation}

Since we have already known that $A_k$ is an isomorphism, and $B$ is compact (see e.g. Chapter 4 in \cite{CCM11}), the Fredholm alternative could be applied to equations \eqref{eq:VarFrame} and \eqref{eq:EleMedSys2VF}. Now, we need to prove that the operator equation \eqref{eq:EleMedSys2VF} is solvable in the quotient space $H_0^2(\mathrm{curl}, D)/\mathbb{U} $, where $\mathbb{U}$ is the eigenspace of the system \eqref{eq:VarFrame}. We denote by $\bm{f} \in H_0^2(\mathrm{curl}, D)$ the Riesz representation of $\mathfrak{f}$ in $H_0^2(\mathrm{curl}, D)$. Next, we illustrate that
\[
\bm{f} \in \mathrm{ker}[(A_k-k^2B)^*]^{\bot}
\]
where $(A_k-k^2B)^*$ is adjoint to $A_k-k^2B$. If $\bm{v} \in \mathrm{ker}(A_k-k^2B)^*$, then for any $\bm{w} \in H^2(\mathrm{curl}, D)$, there is
\[
\left(\bm{w}, (A_k-k^2B)^* \bm{v} \right)_{H^2(\mathrm{curl}, D)}=0,
\]
which means
\begin{equation}\label{eq:ker_con}
\left((\bm{N}-\bm{I})^{-1}(\text{curl curl}-k^2)\bm{w},(\text{curl curl}-k^2)\bm{v} \right)_D +k^4( \bm{w}, \bm{v})_D-k^2(\text{curl}\, \bm{w}, \text{curl}\, \bm{v})_D=0.
\end{equation}
Then we shall show that $(\bm{v},\bm{f})_{H^2(\mathrm{curl}, D)}=0$. Since $\bm{f}$ is the Riesz representation of $\mathfrak{f}$, applying \eqref{eq:ker_con}, one has
 \begin{equation*}\label{eq:vf}
 \begin{split}
(\bm{v},&\bm{f})_{H^2(\mathrm{curl}, D)}\\
&=\mathfrak{f}(\bm{v})\\
&=\left(-(\bm{N}-\bm{I})^{-1}(\text{curl curl} -k^2 \bm{N}(\bm{x}))\bm{\zeta},(\text{curl curl}-k^2)\bm{v}\right)_D\\
&=-\left((\bm{N}-\bm{I})^{-1}(\text{curl curl} -k^2)\bm{\zeta}-(\bm{N}-\bm{I})^{-1}k^2(\bm{N}-\bm{I})\bm{\zeta},(\text{curl curl}-k^2)\bm{v} \right)_D\\
&=-\left((\bm{N}-\bm{I})^{-1}(\text{curl curl} -k^2)\bm{\zeta},(\text{curl curl}-k^2)\bm{v} \right)_D+(k^2\bm{\zeta}, (\text{curl curl}-k^2)\bm{v} )_D\\
&=-\left((\bm{N}-\bm{I})^{-1}(\text{curl curl} -k^2)\bm{\zeta},(\text{curl curl}-k^2)\bm{v} \right)_D+k^2(\text{curl}\,\bm{\zeta}, \text{curl}\,\bm{v})_D-k^4(\bm{\zeta},\bm{v})_D\\
&=0,
\end{split}
\end{equation*}
where in the second last equality, Green's vector theorem and Gauss's divergence theorem have been used to get
\[
(\bm{\zeta}, \text{curl curl}\,\bm{v} )_D= (\text{curl}\,\bm{\zeta}, \text{curl}\,\bm{v} )_D.
\]
Therefore, the equation \eqref{eq:EleMedSys2VF} is solvable in the quotient space $H_0^2(\mathrm{curl}, D)/\mathbb{U} $. Set
\begin{equation}\label{eq:soln}
\hat{\bm{u}}=(A_k-k^2B)^{-1}\bm{f} \quad \text{in }H_0^2(\mathrm{curl}, D)/\mathbb{U}.
\end{equation}
Noting \eqref{eq:zeta_small} and \eqref{eq:anti_f}, one obtains
\begin{equation}\label{eq:f_small}
\|\bm{f} \|_{(L^2(D))^3}\leq C\psi(\epsilon),
\end{equation}
where $C$ is a constant depending on $\bm{N}$, $k$ and $D$. Combining \eqref{eq:zeta_small}, \eqref{eq:soln} and \eqref{eq:f_small}, one has
\begin{equation}\label{eq:app_small}
\begin{split}
\|\bm{E}_{0,k}- \bm{E}_{\bm{g}_0,k}+\bm{E}^t-\bm{E}_k \|_{ (L^2(D))^3} \leq &\| \widetilde{\bm{u}}-\bm{u} \|_{  (L^2(D))^3}+\|\bm{\zeta} \|_{(L^2(D))^3 }\\
\leq & C\psi(\epsilon) \rightarrow 0\ \ \mbox{as}\ \ \epsilon\rightarrow+0.
\end{split}
\end{equation}

Now, we analyze the interior transmission eigenvalue problem by studying the function $\bm{v}$ instead of $\bm u$ defined in \eqref{eq:u_v}. Taking the difference between the first two equations in \eqref{eq:elec_TEV_Pro}, with the help of \eqref{eq:u_v}, one obtains
\begin{equation}\label{eq:diff_u_v}
\text{curl curl}\,\bm{u}=k^2 \bm{v}\quad \text{in}\ D.
\end{equation}
Next, we define the spaces
\[
S_i:=\{  \bm{f} \in H^{i}(\mathrm{curl},D)| \mbox{div} \bm{f} = 0\},\ i= 0,1,
\]
and the mapping
\[
\mathrm{curl}^{-1}: S_0 \rightarrow S_1/S'
\]
that maps a vector field to its vector potential. Here, $S':= \{  \bm{f} \in H(\mathrm{curl},D)| \mbox{curl} \bm{f}=0,\ \mbox{div} \bm{f}= 0\}$. Besides, we define another mapping that maps a vector field to its vector potential,
\[
\widetilde{\mathrm{curl}}^{-1}:  S_1/S'  \rightarrow H^2(\mathrm{curl},D)/S'',
\]
where $S'':= \{  \bm{f} \in H^2(\mathrm{curl},D)| \mbox{curl}\bm{f}=0\}$. We then represent these two mappings indiscriminately when the space is clear. Accordingly, we define the inner product space
\[
V_0:=\{\bm{f}\in S_0|  \mathrm{curl}^{-1}\mathrm{curl}^{-1}\bm{f} \in H_0^2(\mathrm{curl},D) \}
\]
equipped with the $(L^2(D))^3 $ scalar product. Therefore, the existence and uniqueness of a strong solution to \eqref{eq:elec_TEV_Pro} is equivalent to the existence and uniqueness of $\bm{v}\in V_0$ and $\bm{u}\in H_0^2(\mathrm{curl},D)$ satisfying \eqref{eq:diff_u_v} and
\begin{equation}\label{eq:pde_v}
(\text{curl curl}-k^2\bm{N} )(\bm{N}-\bm{I})^{-1}(\bm{v}-k^2\mathrm{curl}^{-1}\mathrm{curl}^{-1}\bm{v} )=0 \quad \text{in }D.
\end{equation}
After using the variational formulation, one easily see that $\bm{v} \in  V_0$ satisfies \eqref{eq:pde_v} if and only if
\begin{equation}\label{eq:var_v}
\left((\bm{N}-\bm{I})^{-1}(\bm{v}-k^2 \mathrm{curl}^{-1}\mathrm{curl}^{-1}\bm{v}),\bm{v}'-k^2\bm{N}\mathrm{curl}^{-1}\mathrm{curl}^{-1} \bm{v}' \right)_D=0
\end{equation}
for all $\bm{v}' \in V_0$.

Then, we set
\begin{equation}\label{eq:ses_G}
\begin{split}
\mathcal{G}_k(\bm{v},\bm{v}')=&\left((\bm{N}-\bm{I})^{-1}(\bm{v}-k^2\mathrm{curl}^{-1}\mathrm{curl}^{-1} \bm{v}), \bm{v}'-k^2\mathrm{curl}^{-1}\mathrm{curl}^{-1}\bm{v}' \right)_D\\
&+ \left(k^2\mathrm{curl}^{-1}\mathrm{curl}^{-1} \bm{v}, k^2\mathrm{curl}^{-1}\mathrm{curl}^{-1} \bm{v}' \right)_D
\end{split}
\end{equation}
and
\begin{equation}\label{eq:ses_J}
\mathcal{J}(\bm{v},\bm{v}')=\left(\mathrm{curl}^{-1}\bm{v}, \mathrm{curl}^{-1}\bm{v}' \right)_D,
\end{equation}
which define two sesquilinear forms on $ V\times V$. Using the identities
\[
\bm{N}(\bm{N}-\bm{I})^{-1}=\bm{I}+(\bm{N}-\bm{I})^{-1}
\]
and
\[
\left(\bm{v}, \mathrm{curl}^{-1}\mathrm{curl}^{-1}\bm{v}' \right)_D=\left(\mathrm{curl}^{-1}\bm{v}, \mathrm{curl}^{-1}\bm{v}' \right)_D
\]
for all $(\bm{v},\bm{v}')\in V\times V_0$, the equation \eqref{eq:var_v} can be rewritten as
\begin{equation}\label{eq:pro_sesqui}
\mathcal{G}_k(\bm{v},\bm{v}')-k^2\mathcal{J}(\bm{v},\bm{v}')=0
\end{equation}
for all $\bm{v}' \in V_0$. With the help of Poincar\'e-type inequality \cite{Monk}, we know that the sesquilinear forms $\mathcal{G}_k$ and $\mathcal{J}$ are bounded. Based on the Riesz representation theorem, we could define the continuous operators $G_k : V_0\to V_0$ such that
\[
(G_k \bm{v}_0,\bm{v}')_D =\mathcal{G}_k(\bm{v}_0,\bm{v}')
\]
for all $\bm{v}' \in V_0$, and $J: V_0\to V_0$ such that
\[
(J\bm{v}_0,\bm{v}')_D =\mathcal{J}(\bm{v}_0,\bm{v}')
\]
for all $\bm{v}' \in V_0$.

Next, we show that $\mathcal{G}_k$ is a coercive sesquilinear form on $ V_0\times V_0$. According to \eqref{eq:mat_posi_def}, there is
\begin{equation}\label{eq:G_1}
\mathcal{G}_k(\bm{v}_0,\bm{v}_0)\geq \alpha \|\bm{v}_0-k^2\mathrm{curl}^{-1}\mathrm{curl}^{-1} \bm{v}_0\|^2_{(L^2(D))^3 }+k^4\|\mathrm{curl}^{-1}\mathrm{curl}^{-1} \bm{v}_0 \|^2_{(L^2(D))^3 }.
\end{equation}
Setting $a= \| \bm{v}_0 \|_{(L^2(D))^3 }$ and $b=k^2\| \mathrm{curl}^{-1}\mathrm{curl}^{-1} \bm{v}_0\|_{(L^2(D))^3 }$, we have
\begin{equation}\label{eq:G_2}
\|\bm{v}_0-k^2\mathrm{curl}^{-1}\mathrm{curl}^{-1} \bm{v}_0\|^2_{(L^2(D))^3 }\geq a^2+b^2-2ab.
\end{equation}
Combining \eqref{eq:G_1} and \eqref{eq:G_2}, one concludes that
\begin{equation*}\label{eq:G_3}
\mathcal{G}_k(\bm{v}_0,\bm{v}_0)\geq \alpha a^2-2\alpha ab+(1+\alpha)b^2\geq \frac{\alpha}{1+2\alpha}(a^2+b^2)\geq \frac{\alpha}{1+2\alpha} \|\bm{v}_0 \|^2_{(L^2(D))^3 }.
\end{equation*}
Since $\mathcal{G}_k$ is a coercive sesquilinear form on $ V_0\times V_0$, by the Lax–Milgram theorem, we deduce that $G_k : V_0\to V_0$ is a bijective operator.

Regarding $J$, we would like to illustrate that it is compact. Let $\{\bm{v}_n \}$ be a bounded sequence in $V_0$, by the Bolzano-Weierstrass theorem, there exists a subsequence, denoted by $\{\bm{v}_{n_i} \}$, that converges weakly to some $\bm{v}_0$ in $V_0$. Provided $\partial D$ is sufficiently smooth, the space
\[
\{\bm{u}\in H_0(\mathrm{curl},D):\text{div}\,\bm{u}=0  \ \mathrm{in} \ D \}
\]
is continuously embedded into $H^1(D)$. So, the sequence $\{ \mathrm{curl}^{-1} \bm{v}_{n_i}\} $ is bounded in $H^1(D)$. By the Rellich compact embedding theorem, one has that $\{ \mathrm{curl}^{-1} \bm{v}_{n_i}\} $ converges strongly to $\mathrm{curl}^{-1} \bm{v}_0$ in $(L^2(D))^3$. According to the definition of $J$ and the Cauchy-Schwarz inequality, there is
\[
\|J(\bm{v}_{n_i}-\bm{v}_0 )  \|^2_{(L^2(D))^3 }\leq \|\mathrm{curl}^{-1}(\bm{v}_{n_i}-\bm{v}_0) \|_{(L^2(D))^3 } \| \mathrm{curl}^{-1}( J(\bm{v}_{n_i}-\bm{v}_0) ) \|_{(L^2(D))^3 }.
\]
Besides, using Poincaré-type inequality, we deduce that
\[
\|J(\bm{v}_{n_i}-\bm{v}_0 )  \|_{(L^2(D))^3 }\geq \| \mathrm{curl}^{-1}( J(\bm{v}_{n_i}-\bm{v}_0) ) \|_{(L^2(D))^3 }.
\]
Combining the above two formulas, we get
\[
\|J(\bm{v}_{n_i}-\bm{v}_0 )  \|_{(L^2(D))^3 }\leq \|\mathrm{curl}^{-1}(\bm{v}_{n_i}-\bm{v}_0) \|_{(L^2(D))^3 }.
\]
Therefore, the sequence $\{J\bm{v}_{n_i} \}$ converges strongly to $J\bm{v}_0$ in $V_0$.


We now consider the scattering system \eqref{eq:EleMedSys} again. We take $\bm{\zeta}\in H^2(\mathrm{curl},D)$ satisfying \eqref{eq:zeta_boundary} and \eqref{eq:zeta_small}, and define $\bm{\iota}:=\text{curl curl}\,\bm{\zeta}\in V$. Setting $\bar{\bm{E}}:=\bm{E}^t-\bm{\zeta} \in H^2(\mathrm{curl},D)$, the system \eqref{eq:EleMedSys} becomes
\begin{equation}\label{eq:EleMedSys3}
\left\{
\begin{array}{ll}
\text{curl curl}\,\bar{\bm{E}}-k^2 \bm{N}(\bm{x})\bar{\bm{E}} =-(\text{curl curl} -k^2 \bm{N}(\bm{x}))\bm{\zeta} \quad &\text{in} \ D, \medskip \\
\text{curl curl}\,\bm{E}_{\bm{g}_0,k}-k^2 \bm{E}_{\bm{g}_0,k}=0\quad &\text{in} \ D, \medskip \\
\bm{\nu}\times \bar{\bm{E}}=\bm{\nu}\times \bm{E}_{\bm{g}_0,k}  \quad &\text{on} \ \partial D, \medskip \\
\bm{\nu}\times \text{curl}\,\bar{\bm{E}} =\bm{\nu}\times \text{curl}\,\bm{E}_{\bm{g}_0,k} \quad &\text{on} \ \partial D.
\end{array}
\right.
\end{equation}
Let $\bar{\bm{u}} :=\bar{\bm{E}}- \bm{E}_{\bm{g}_0,k} \in H_0^2(\mathrm{curl},D)$, $\bar{\bm{v}}:=\bm{N}(\bm{x})\bar{\bm{E}}-\bm{E}_{\bm{g}_0,k}-k^{-2}\text{curl curl}\,\bm{\zeta} +\bm{N}(\bm{x})\bm{\zeta} \in V_0$. Then $\text{curl curl}\,\bar{\bm{u}}=k^2\bar{\bm{v}}$, and the equations in \eqref{eq:EleMedSys3} can be rewritten as
\begin{equation}\label{eq:inhomo_v_eta}
\begin{split}
(\text{curl curl}-k^2 \bm{N})(\bm{N}-\bm{I})^{-1}(\bar{\bm{v}}-k^2\mathrm{curl}^{-1}\mathrm{curl}^{-1}\bar{\bm{v}}+&k^{-2}\bm{\iota}-\bm{N}\mathrm{curl}^{-1}\mathrm{curl}^{-1}\bm{\iota} )\\
&+\bm{\iota}-k^2\bm{N}\mathrm{curl}^{-1}\mathrm{curl}^{-1}\bm{\iota} =0\quad \text{in }D.
\end{split}
\end{equation}
Using a variational framework, we can see that $\bar{\bm{v}}\in V_0$ satisfying \eqref{eq:inhomo_v_eta} is equivalent to
\begin{equation}\label{eq:var_v_eta}
\mathcal{G}_k(\bar{\bm{v}},\bm{v}')-k^2\mathcal{J}(\bar{\bm{v}},\bm{v}')=\mathfrak{q}(\bm{v}')\quad \text{for all }\bm{v}'\in V_0,
\end{equation}
where $ \mathcal{G}_k$ and $\mathcal{J}$ are sesquilinear forms defined in \eqref{eq:ses_G} and \eqref{eq:ses_J}. And $\mathfrak{q}(\bm{v}')$ is an antilinear form defined by
\begin{equation}\label{eq:anti_q}
\begin{split}
\mathfrak{q}(\bm{v}'):=&\left(\mathrm{curl}^{-1}\mathrm{curl}^{-1}\bm{\iota}+(\bm{N}-\bm{I})^{-1}\mathrm{curl}^{-1}\mathrm{curl}^{-1}\bm{\iota}-k^{-2}(\bm{N}-\bm{I})^{-1}\bm{\iota}, \bm{v}'\right)_D\\
&+\left((\bm{N}-\bm{I})^{-1}\bm{\iota}-k^2\mathrm{curl}^{-1}\mathrm{curl}^{-1}\bm{\iota}-k^2(\bm{N}-\bm{I})^{-1}\mathrm{curl}^{-1}\mathrm{curl}^{-1}\bm{\iota}, \mathrm{curl}^{-1}\mathrm{curl}^{-1}\bm{v}' \right)_D.
\end{split}
\end{equation}

As $G_k$ is an isomorphism, and $J$ is compact, the Fredholm alternative can be applied to the systems \eqref{eq:pro_sesqui} and \eqref{eq:var_v_eta}. Now, we show that \eqref{eq:var_v_eta} is solvable in the quotient space $V_0/\mathbb{V} $, where $\mathbb{V} $ is the eigenspace of \eqref{eq:pro_sesqui}. We denote by $\bm{\tilde q}$ the Riesz representation of $\mathfrak{q} $ in $V_0$. Next, we verify that
\[
\bm{\tilde q}\in \text{ker}[(G_k-k^2J)^*]^{\bot}.
\]
If the function $\bm{\varsigma}\in V_0$ belongs to the kernel of $(G_k-k^2J)^*$, then for any $\bm{\tau}\in V$, one has
\[
\left(\bm{\tau}, (G_k-k^2J)^*\bm{\varsigma} \right)_D=0,
\]
that is
\begin{equation}\label{eq:ker_tau_sigma}
\begin{split}
&\left((\bm{N}-\bm{I})^{-1}(\bm{\tau}-k^2\mathrm{curl}^{-1}\mathrm{curl}^{-1}\bm{\tau}),\bm{\varsigma}-k^2\mathrm{curl}^{-1}\mathrm{curl}^{-1}\bm{\varsigma} \right)_D\\
&\quad+(k^2\mathrm{curl}^{-1}\mathrm{curl}^{-1}\bm{\tau}, k^2\mathrm{curl}^{-1}\mathrm{curl}^{-1}\bm{\varsigma} )_D-k^2(\mathrm{curl}^{-1}\bm{\tau}, \mathrm{curl}^{-1} \bm{\varsigma} )_D=0.
\end{split}
\end{equation}
As $\bm{\tilde q}$ is the Riesz representation of $\mathfrak{q} $, with the help of \eqref{eq:ker_tau_sigma}, there is
\begin{equation*}\label{eq:ortho_sig_q}
\begin{split}
(\bm{\varsigma}, \bm{\tilde q}&)_D\\
=&\mathfrak{q} (\bm{\varsigma} )\\
=&\left(\mathrm{curl}^{-1}\mathrm{curl}^{-1}\bm{\iota}+(\bm{N}-\bm{I})^{-1}\mathrm{curl}^{-1}\mathrm{curl}^{-1}\bm{\iota}-k^{-2}(\bm{N}-\bm{I})^{-1}\bm{\iota}, \bm{\varsigma} \right)_D\\
&+\left((\bm{N}-\bm{I})^{-1}\bm{\iota}-k^2\mathrm{curl}^{-1}\mathrm{curl}^{-1}\bm{\iota}-k^2(\bm{N}-\bm{I})^{-1}\mathrm{curl}^{-1}\mathrm{curl}^{-1}\bm{\iota}, \mathrm{curl}^{-1}\mathrm{curl}^{-1}\bm{\varsigma} \right)_D\\
=&0.
\end{split}
\end{equation*}
Up to now, we have verified that \eqref{eq:var_v_eta} is solvable in the quotient space $V_0/\mathbb{V} $. By \eqref{eq:zeta_small} and \eqref{eq:anti_q}, we observes that
\begin{equation}\label{eq:q_small}
\|  \bm{\tilde q}\|_{(L^2(D))^3}\leq C\psi(\epsilon),
\end{equation}
where $C$ is a constant depending on $\bm{N}, k$ and $D$. Setting
\begin{equation*}\label{eq:sol_v}
\hat{\bm{v}}=(G_k-k^2J)^{-1}\bm{\tilde q}  \quad \text{in }V_0/\mathbb{V},
\end{equation*}
and using \eqref{eq:zeta_small}, \eqref{eq:q_small}, one has
\begin{equation}\label{eq:diff_hat_v}
\begin{split}
\|\bm{N}\bm{E}^t- \bm{N} \bm{E}_k-\bm{E}_{\bm{g}_0,k}+\bm{E}_{0,k}\|_{(L^2(D))^3 }\leq &\|\hat{\bm{v}} \|_{(L^2(D))^3}+k^{-2}\| \text{curl curl}\,\bm{\zeta}\|_{(L^2(D))^3}\\
\leq &C\psi(\epsilon)\to 0 \text{  as  }\epsilon\to +0.
\end{split}
\end{equation}
Finally, by combining \eqref{eq:app_small} and \eqref{eq:diff_hat_v}, one obtains
\begin{equation*}
\|\bm{E}_{\bm{g}_0,k}-\bm{E}_{0,k} \|_{(L^2(D))^3}\leq C\psi(\epsilon)\to 0 \text{  as  }\epsilon\to +0.
\end{equation*}

The proof is complete.
\end{proof}

\begin{rem}
It is remarked that Lemma ~\ref{le:Rell} is used in the proof of Theorem ~\ref{th:IndeedAppro}. So, the higher requirement for the regularity of the scatterer is needed technically. However, we believe that this regularity assumption can be relaxed, which is showed in the following numerical examples. The approximation result Theorem ~\ref{th:IndeedAppro} is an extension of the similar result in the acoustic scattering problem, but this study is more technical and challenging. Meanwhile, we believe that this result is not only of independent interest to the scattering theory of electromagnetic waves, but also provides a theoretical basis for practical applications, such as super-resolution imaging and invisibility cloaking.
\end{rem}

\begin{rem}
Theorem \ref{th:EigenfE} and \ref{th:IndeedAppro} indicate that if the Herglotz kernel $\bm g_0$ is a solution to the optimization problem \eqref{eq:AlterOptnEle}, then the Herglotz wave function $\bm{E}_{\bm{g}_0,k}$ \eqref{eq:ElecHerzPair} is approximated to the transmission eigenfunction $\bm{E}_{0,k}$ associated with the correspondingly eigenvalue $k$.
Thus, in order to recover the transmission eigenfunction, it is critical for solving the constrained optimization problem \eqref{eq:AlterOptnEle} and obtain a satisfactory Herglotz kernel $\bm g_0$. In what follows, we will use the Fourier-total-least-square (FTLS) method and Gradient-total-least-square (GTLS) method \cite{LLWW} to compute $\bm g_0$, respectively.
\end{rem}

\section{Imaging of the scatterer}

According to sections \ref{sec:2} and \ref{sec:3}, we can determine the transmission eigenvalues and the corresponding  transmission eigenfunctions from the electric far-field data in  \eqref{eq:nonlin_ope_eqn}. In this section, we show that  the ``invisible patterns'', the transmission eigenfunctions, can be used to qualitatively image the shape of the electromagnetic medium scatterer $D$, independent of $\bm{N}$.

In our recent studies, we find that the Maxwell transmission eigenfunctions $\bm{E}_{0,k}$ possess the following interesting global and local geometric properties. From a global perspective, there exists a sequence of surface-localized transmission eigenfunctions associated with sufficient large transmission eigenvalues, that is,  the $L^2$-energy of the eigenfunction concentrates on the boundary $\partial D$ \cite{DLWW}.
We would like to point out that these surface-localized eigenstates occur very frequently through theoretical and numerical observations.  From a local perspective, the transmission eigenfunctions are nearly vanishing around the high-curvature points on $\partial D$, especially around the corner points of $\partial D$ \cite{BlastenLX}. Meanwhile, the rigorous mathematical justifications in section \ref{sec:3} show that the Herglotz wave $\bm{E}_{\bm{g}_0,k}$ is an approximation of  the transmission eigenfunction $\bm{E}_{0,k}$.
 So the Herglotz wave has the same geometric  attributes as the transmission eigenfunction.
 Hence, based on the global and local properties of the Herglotz wave,  we introduce the following indicator function for identifying the shape of the scatterer $D$:
 \begin{equation*}\label{eq:single-indicator}
I_k^{\text{Res}}(\bm{z}):= -\text{ln}\| \bm{E}_{\bm{g}_0,k}(\bm{z}) \|,
\end{equation*}
where $\bm z\in \widetilde{D}$ is the sampling point with $D\subset\widetilde{D}\subset \mathbb{R}^3$. It is worth noting that the wave field $\bm{E}_{\bm{g}_0,k}$ is nearly vanishing in the interior of $D$  and it is vanishing around the high-curvature points. Thus, the indicator function $I_k^{\text{Res}}(\bm{z})$ possesses a relative large value if $\bm{z}$ belongs to the interior of $D$ or locates at the corner/edge/highly-curved place on $\partial D$ \cite{DLWY}, whereas it possesses a relatively small value if $\bm{z}$ locates in the other places around $\partial D$ and outside of $D$.

 Noting that the multiple transmission eigenfunctions possess the same geometric properties, we superimpose the imaging effects by introducing the following multi-frequency imaging functional:
\begin{equation}\label{eq:multi-indicator}
I_{\mathbb{K}_L}^{\text{Res}} (\bm{z}):= -\text{ln} \sum\limits_{k\in{\mathbb{K}_L}}| \bm{E}_{\bm{g}_0,k}(\bm{z}) |,
\end{equation}
where $\mathbb{K}_L=\{k_1,k_2,\cdots,k_L \}$ denotes the set of $L$ distinct transmission eigenvalues. Based on the imaging functional \eqref{eq:multi-indicator}, we then propose the following imaging scheme, which is referred to as imaging by interior resonant modes.

\begin{table}[htbp]
\begin{tabular}{p{1.5cm} p{12cm}}
\toprule
\multicolumn{2}{l}{\textbf{Algorithm \uppercase\expandafter{\romannumeral2}:}\ Imaging by interior resonant modes}\\
\midrule
Step 1 & For each resonant wavenumber $k$ found in Algorithm I, solve the optimization problem \eqref{eq:AlterOptnEle} by the FTLS method or the GTLS method \cite{LLWW} to obtain the vector Herglotz kernel $\bm{g}_0$.\\
Step 2 & Calculate the electric Herglotz wave function $\bm{E}_{\bm{g}_0,k}$ with the Herglotz kernel $\bm{g}_0$ by the definition \eqref{eq:ElecHerzPair}.\\
Step 3 & Plot the indicator function \eqref{eq:multi-indicator} in a proper domain containing the electromagnetic medium scatterer $D$ and identify the interior and corners (two dimension) or edges (three dimension) as bright points, and other boundary places as dark points in the graph to obtain the shape of the scatterer $D$.\\
\bottomrule
\end{tabular}
\end{table}

\subsection{Transmission eigenvalues reconstruction}

In this part, we present a numerical experiment to verify the validity of Algorithm \uppercase\expandafter{\romannumeral1} for Maxwell transmission eigenvalues reconstruction based on the LSM. We take the regularization parameter $\delta=10^{-5}$. To begin with,
we consider a unit ball, see Figure \ref{fig:ball}(a).  In order to reduce the calculation expense, we use a normal mesh with mesh size $h\approx 0.15$.
To avoid the inverse crime, we use the finite element method (FEM) to compute a pair of electric far-field pattern
$\bm{E}_{\infty}(\hat{\bm{x}}_i,\bm{d}_j, {\bm p}_j^{\theta};k_{s})$ and $\bm{E}_{\infty}(\hat{\bm{x}}_i,\bm{d}_j, {\bm p}_j^{\phi};k_{s})$,
where $i=1,2,\cdots,M$, $j=1,2,\cdots,N$ and $s=1,2,\cdots,S$.
Here, the observation and incident directions are pseudo uniformly distributed on the unit spherical surface.  The artificial far-field data are computed at $100$ observation directions ($M=100$), $100$ incident directions ($N=100$) and $81$ equally distributed wave numbers ($S=81$)  in $[1.1, 1.5]$. Here the
refractive index is given by
\begin{equation*}
  \bm N(\bm x)=\left(
          \begin{array}{ccc}
            16 & 0 & 0 \\
            0 & 16 & 0 \\
            0 & 0 & 16 \\
          \end{array}
        \right).
\end{equation*}

In Figure \ref{fig:ball}(b), the solid blue line shows the value of $\sum_{\ell=1}^3\|\bm{g}_{\delta}(\cdot, \bm{z}_0,\bm{q}_{\ell};k_s) \|_{(L^2(\mathbb{S}^2))^3}^2$  against $k_s\in [1.1, 1.5]$ for a fixed test point $\bm{z}_0=(0.1, 0.3, 0.4)$ with $1\%$ noise.
As expected, it can be seen that the solid blue line has clear spikes, which indicate the locations of the Maxwell transmission eigenvalues.  To show the accuracy of the reconstruction, we list the exact eigenvalues,  the reconstructed eigenvalues by the mixed FEM \cite{Monk2012} and the LSM, {\color{blue}in} Table \ref{tab1:eigenvalues}. From the results of the comparison, we can see that Algorithm \uppercase\expandafter{\romannumeral1} is valid for identifying the Maxwell transmission eigenvalues.

\begin{figure}
\subfigure[a unit ball centered at $(0,0,0)$]
{\includegraphics[width=0.39\textwidth]{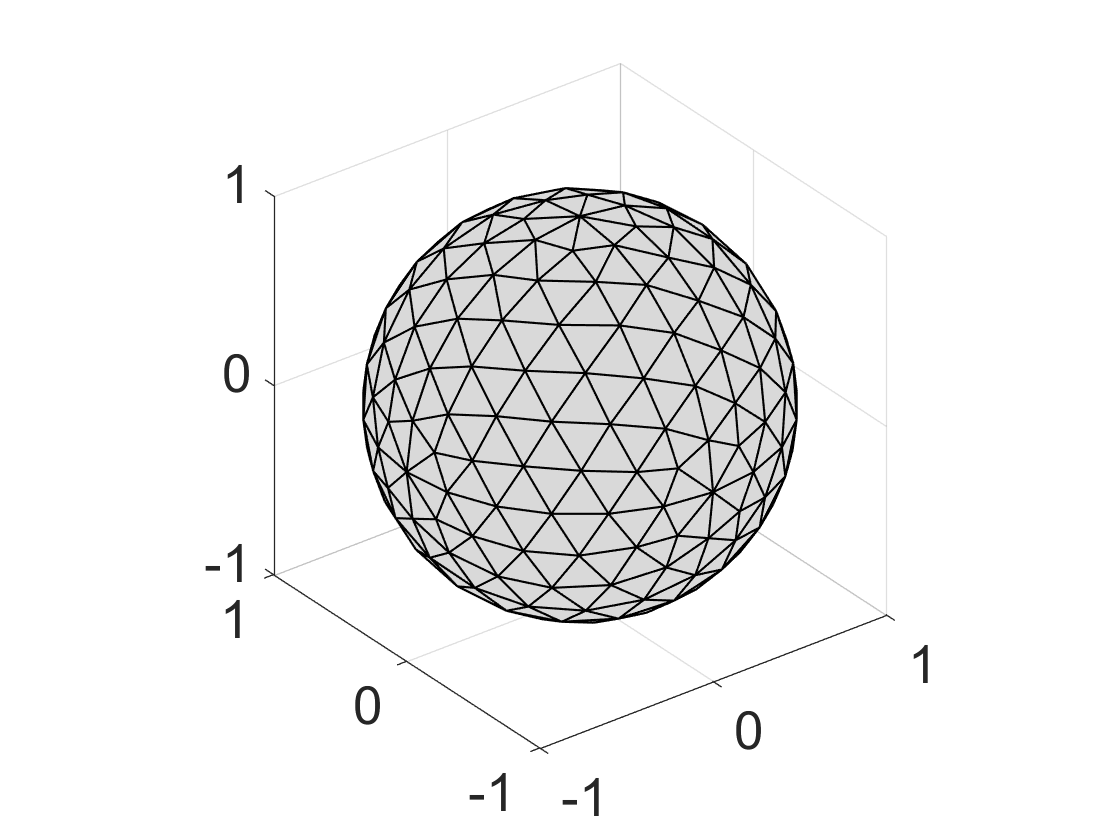}}%
\subfigure[the indicator function vs. $k$]
{\includegraphics[width=0.6 \textwidth]{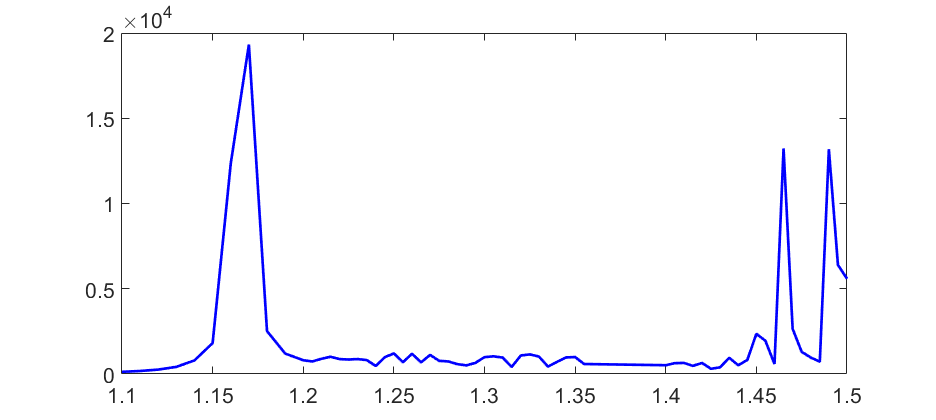}}\\
\caption{\label{fig:ball} (a) $D$: a unit ball; (b) plot of the indicator function $\sum_{\ell=1}^3\|\bm{g}_{\delta}(\cdot, \bm{z}_0,\bm{q}_{\ell};k_s) \|_{(L^2(\mathbb{S}^2))^3}^2$  for a unit ball centered at $(0, 0, 0)$, where $k_s\in [1.1,1.5]$.}
\end{figure}

\begin{table}[h]
\center
 \caption{The first three real Maxwell transmission eigenvalues of the spherical domain. }\label{tab1:eigenvalues}
 \begin{tabular}{lccccccccc}
  \toprule
  Index of eigenvalue  & $1$     & $2$     &$3$         \\
  \midrule
 Exact:                &$1.165$       &$1.461 $       &$1.475 $    \\
 FEM:                 &$1.166$        &$1.462$         &$1.476$     \\
 LSM:                 &$1.170$       &$1.465$        &$1.490$       \\
  \bottomrule
 \end{tabular}
\end{table}

\subsection{Scatterer reconstruction}

In this part, we give several numerical examples to show that the proposed imaging scheme can effectively reconstruct the scatterer $D$.

\begin{figure}
\hfill\subfigure[Gourd]{\includegraphics[width=0.35\textwidth]
                   {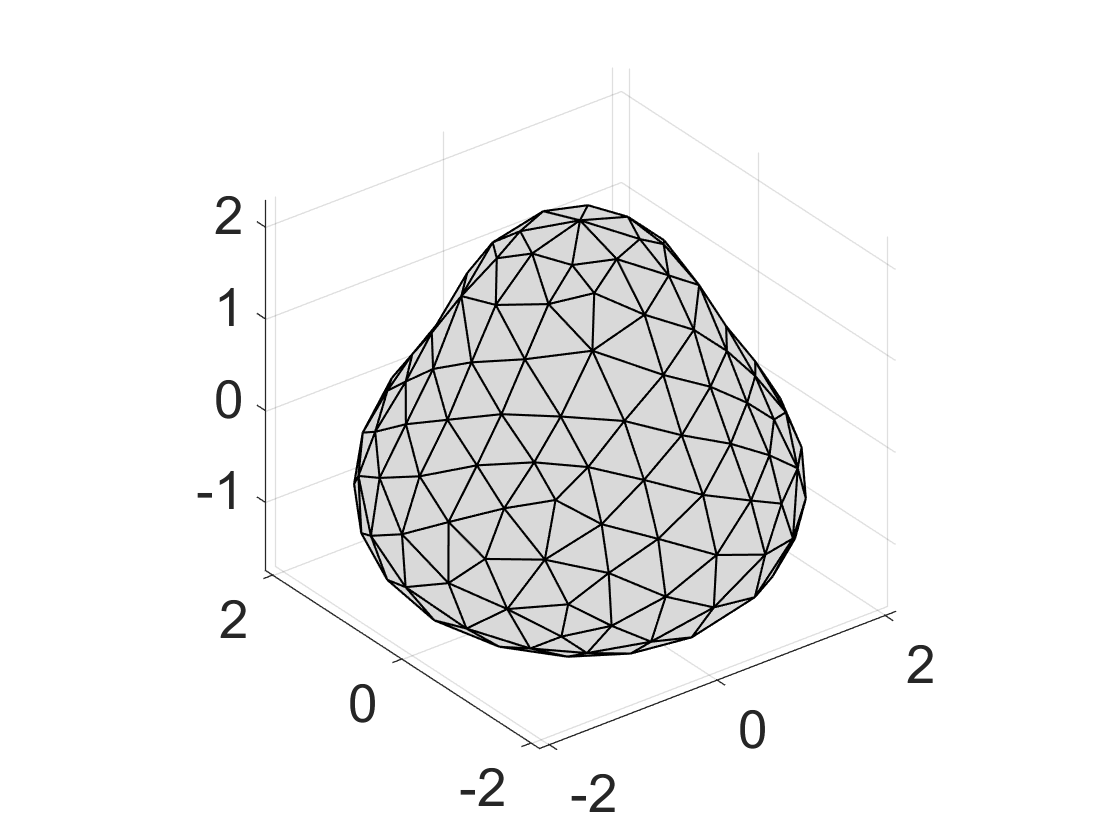}}\hfill
 \hfill\subfigure[Kite]{\includegraphics[width=0.32\textwidth]
                   {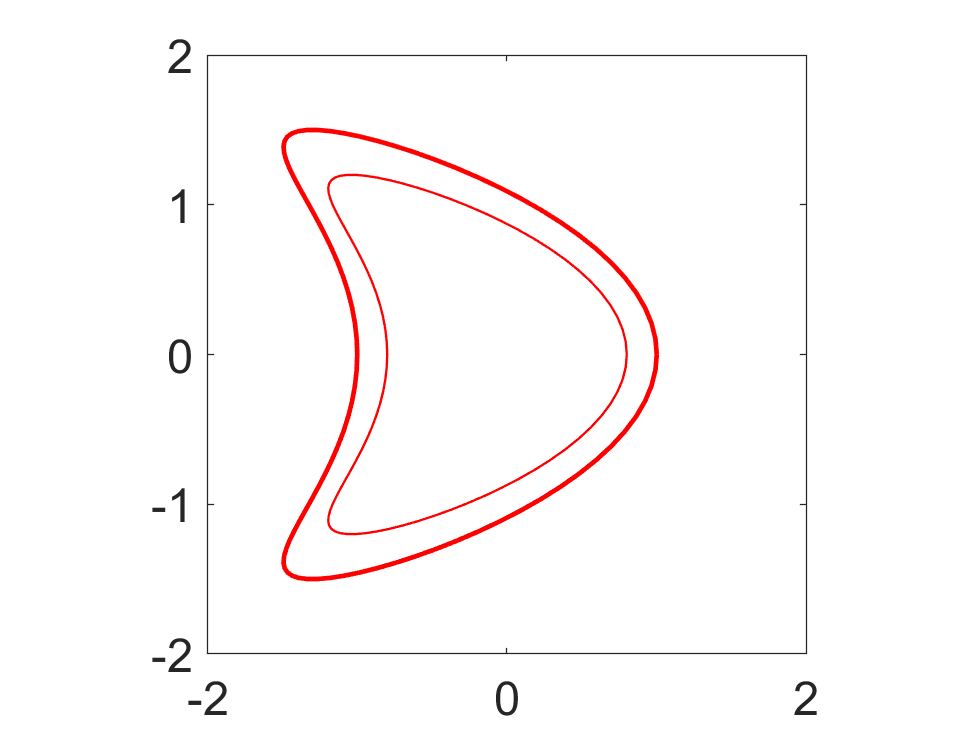}}\hfill
\hfill\subfigure[Square]{\includegraphics[width=0.32\textwidth]
                   {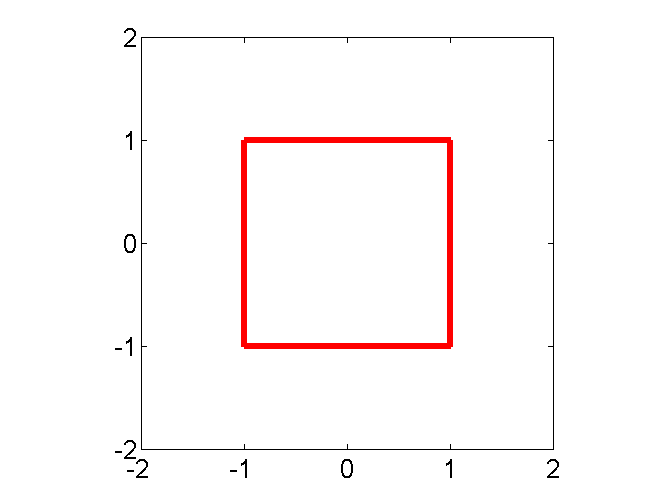}}\hfill
\caption{\label{fig:Geometry} Shapes considered in the numerical examples. (a) gourd-shaped domain in 3D;  (b) layered kite-shaped domain in 2D;  (c) square domain in 2D.}
\end{figure}

\subsubsection{Gourd-shaped domain}

In the first example, we utilize the global properties of transmission eigenfunctions to identify the anisotropic inhomogeneity in the $3$D setting. Let the geometry be a gourd-shaped domain, see Figure \ref{fig:Geometry}(a).
The refractive index is given by
\begin{equation*}
  \bm N(\bm x)=\left(
          \begin{array}{ccc}
            16 & 1 & 0 \\
            1 & 16 & 0 \\
            0 & 0 & 14 \\
          \end{array}
        \right).
\end{equation*}
 The artificial far-field data are computed at $100$ observation directions and $100$ incident directions.
 To begin with, we use the mixed FEM to calculate the transmission eigenfunctions associated with different eigenvalues.
Figure \ref{fig:triangle-eigenfunction} shows the surface plots and multiple slices plots of the eigenfunctions $|\bm E_{0,k}|$ for three different eigenvalues. It is clear to see that the transmission eigenfunctions are surface-localized around the boundary.  In this case, the medium is anisotropic. The surface localization result here extends the theoretical statement of isotropic cases \cite{DLWW} to the anisotropic cases numerically. Besides, it is noted that $|\bm E_{0,k}|$ is nearly vanishing near the top of the gourd-shaped domain. This observation is consistent with the theoretical result that the transmission eigenfunctions must be nearly vanishing near the points of large curvature \cite{Blasten21}.

Next, we test the proposed reconstruction scheme with $1\%$ noise. Here we use the GTLS method with regularization parameter $\beta=10^{-4}$.
Figure \ref{fig:Gourd-MultiIndicator} presents iso-surface plots and slice plots of the multi-frequency indicator function, respectively. Here we use five different frequencies in interval $[1.5, 1.6]$.
In the slice-view plots, the white dashed lined are the boundary of the exact scatterer. One can find that the proposed method could approximately identify the boundary of the scatterer. 

\subsubsection{TE mode in 2D}
\begin{figure}
\hfill\subfigure[$k=1.5141$]{\includegraphics[width=0.33\textwidth]
                   {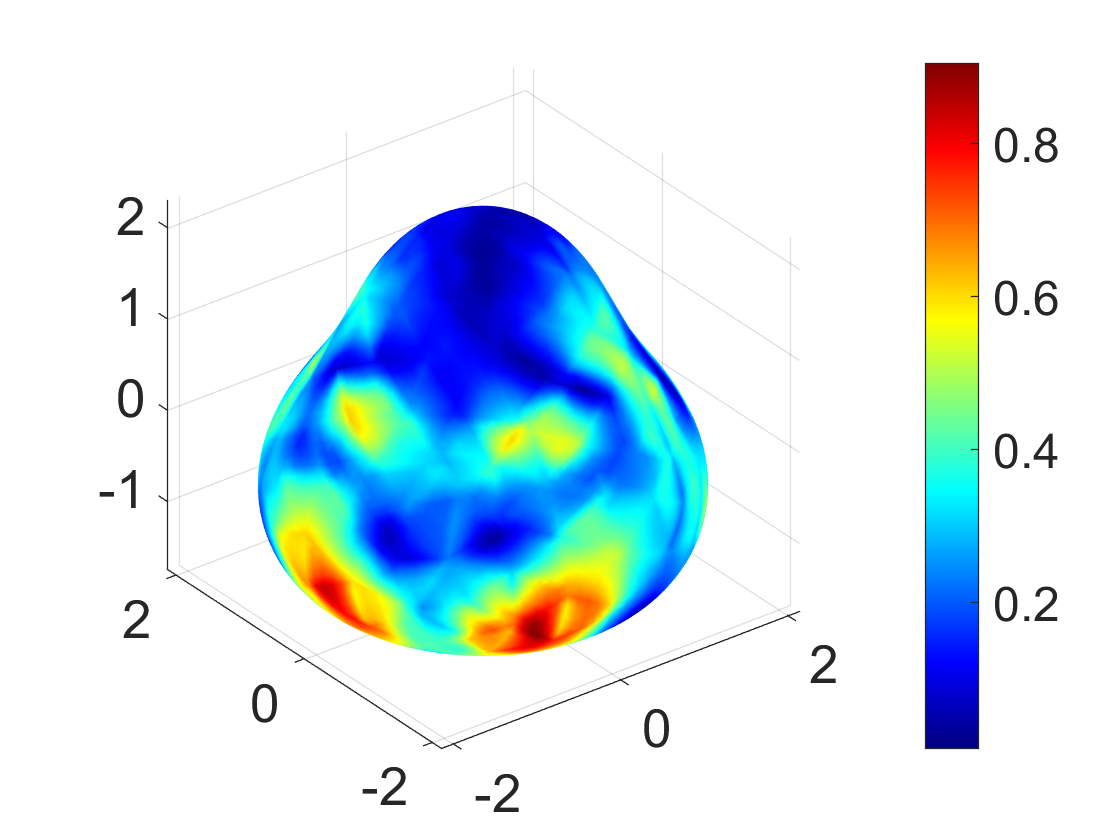}}\hfill
\hfill\subfigure[$k=1.5149$]{\includegraphics[width=0.33\textwidth]
                   {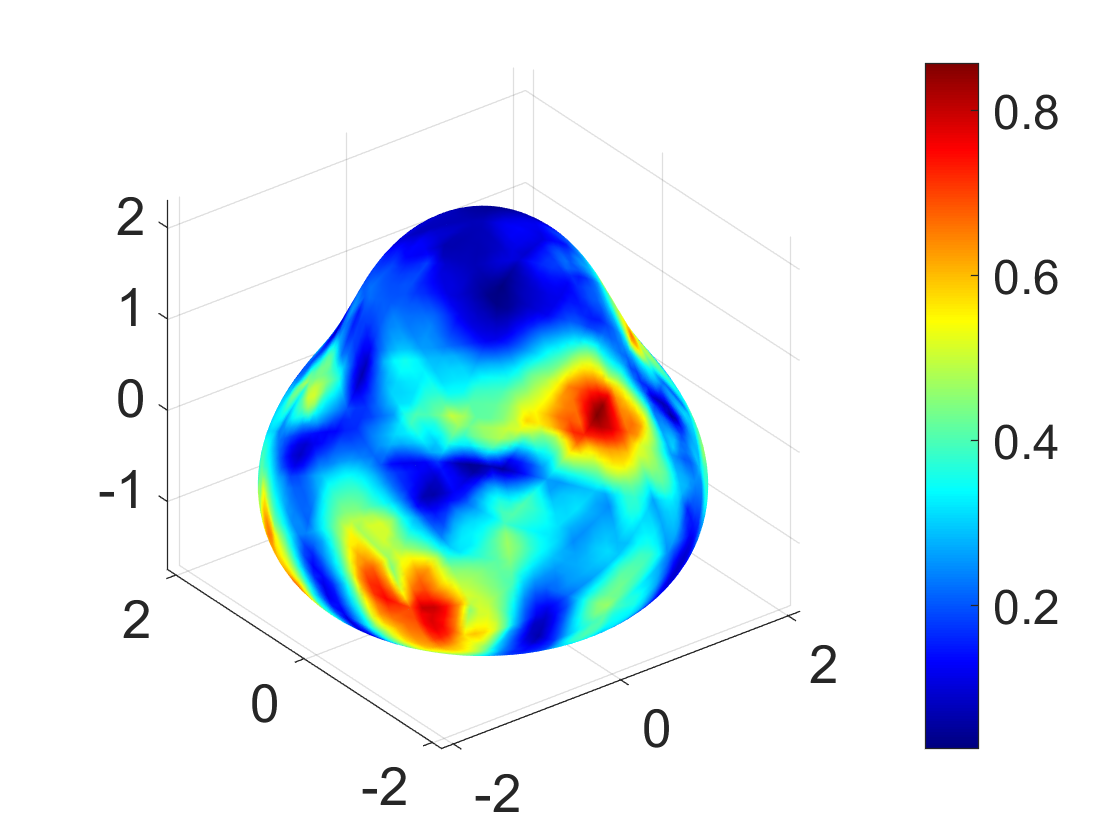}}\hfill
\hfill\subfigure[$k=1.5338$]{\includegraphics[width=0.33\textwidth]
                   {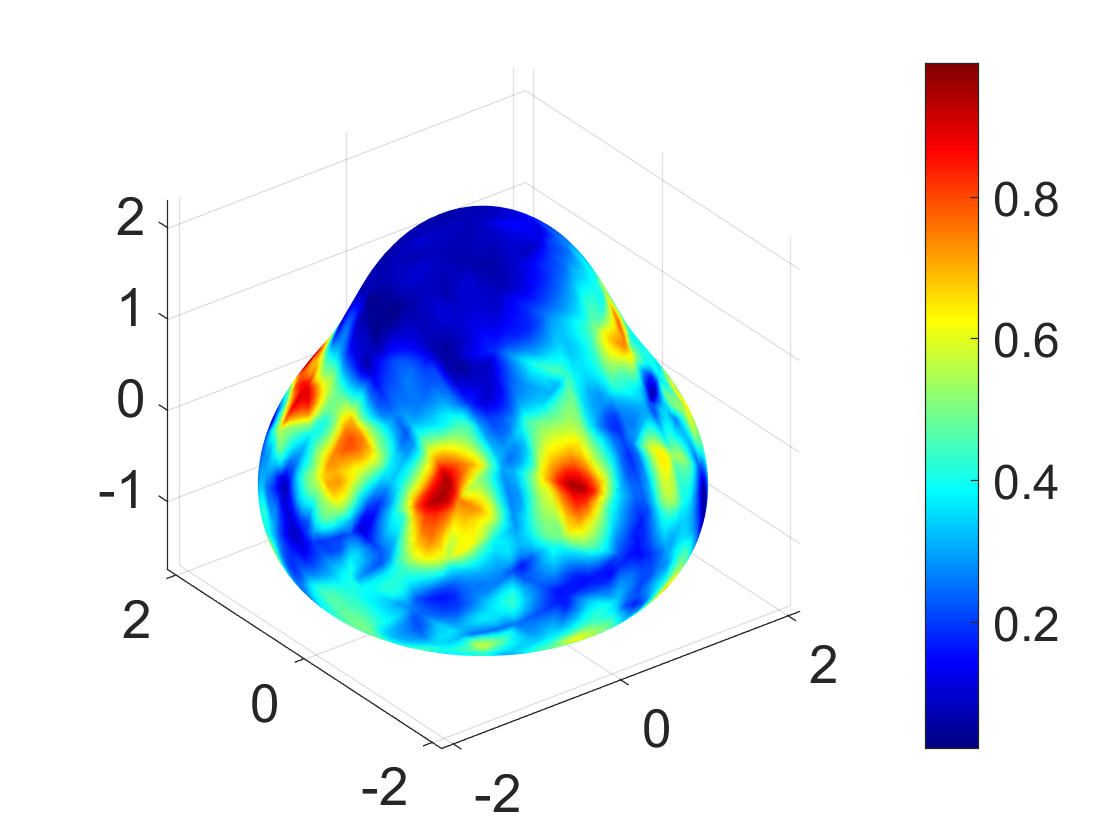}}\hfill\\
 \hfill\subfigure[$k=1.5141$]{\includegraphics[width=0.33\textwidth]
                   {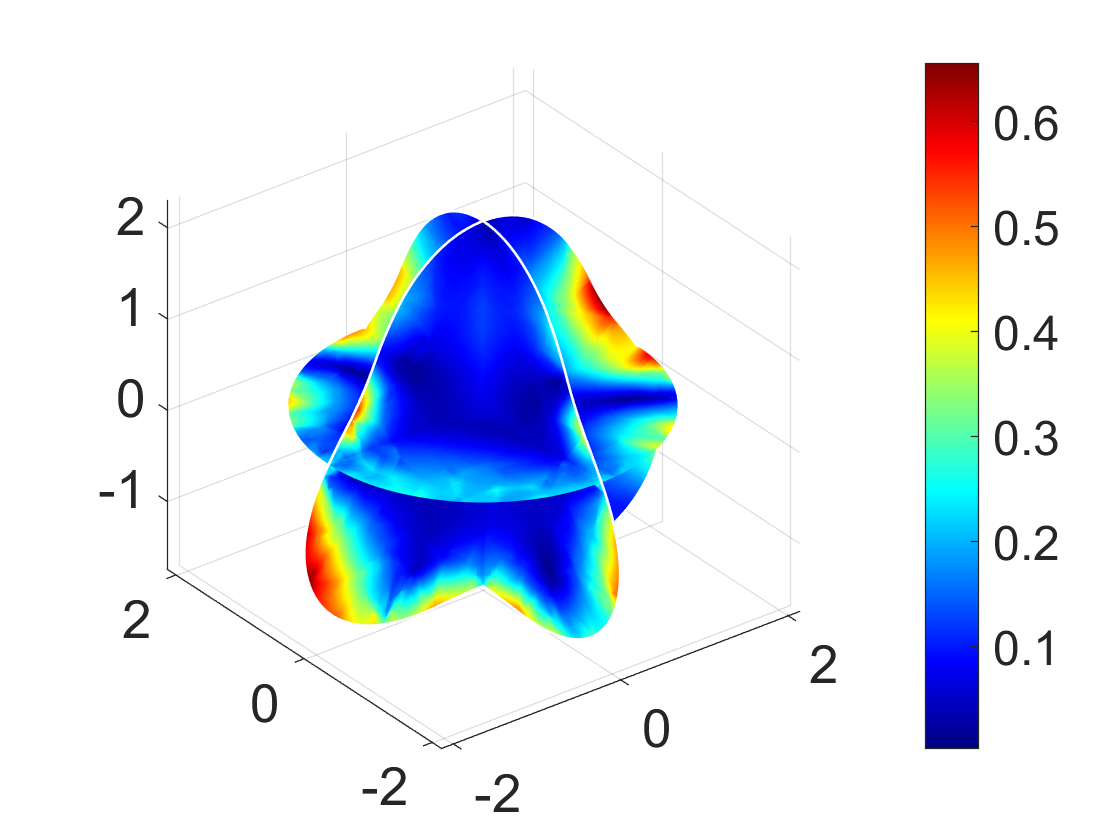}}\hfill
\hfill\subfigure[$k=1.5149$]{\includegraphics[width=0.33\textwidth]
                   {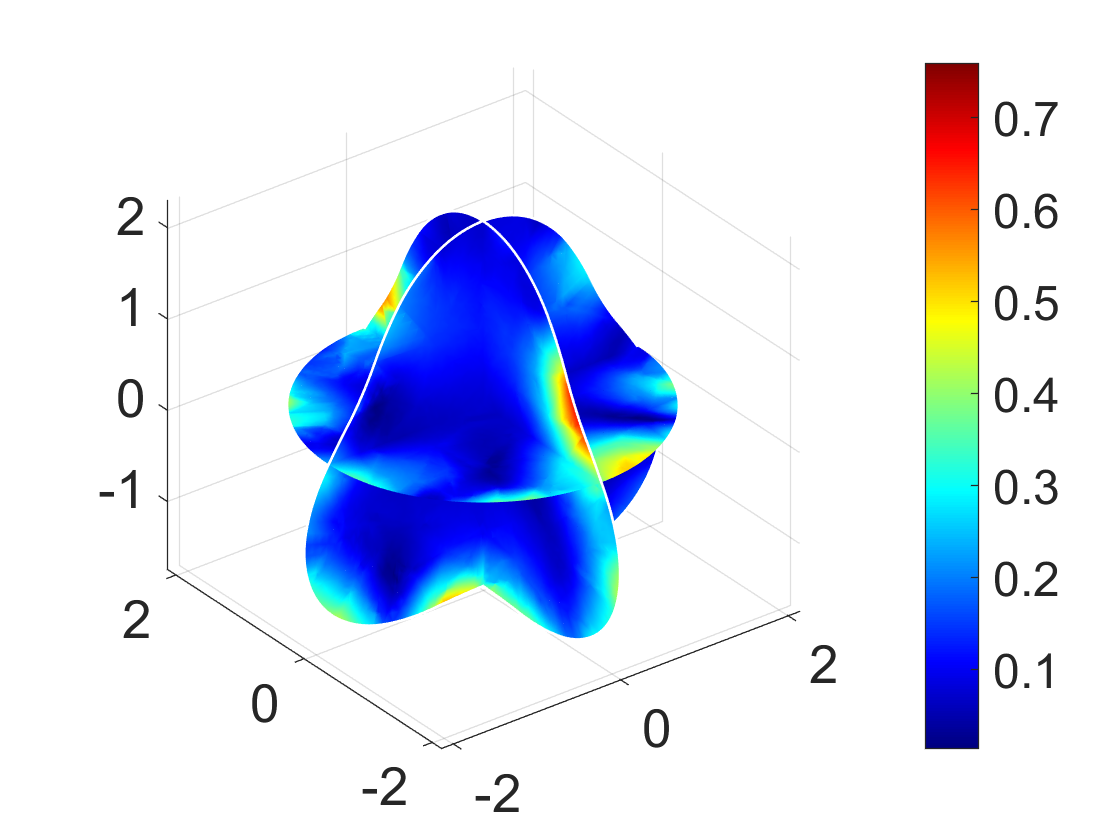}}\hfill
\hfill\subfigure[$k=1.5338$]{\includegraphics[width=0.33\textwidth]
                   {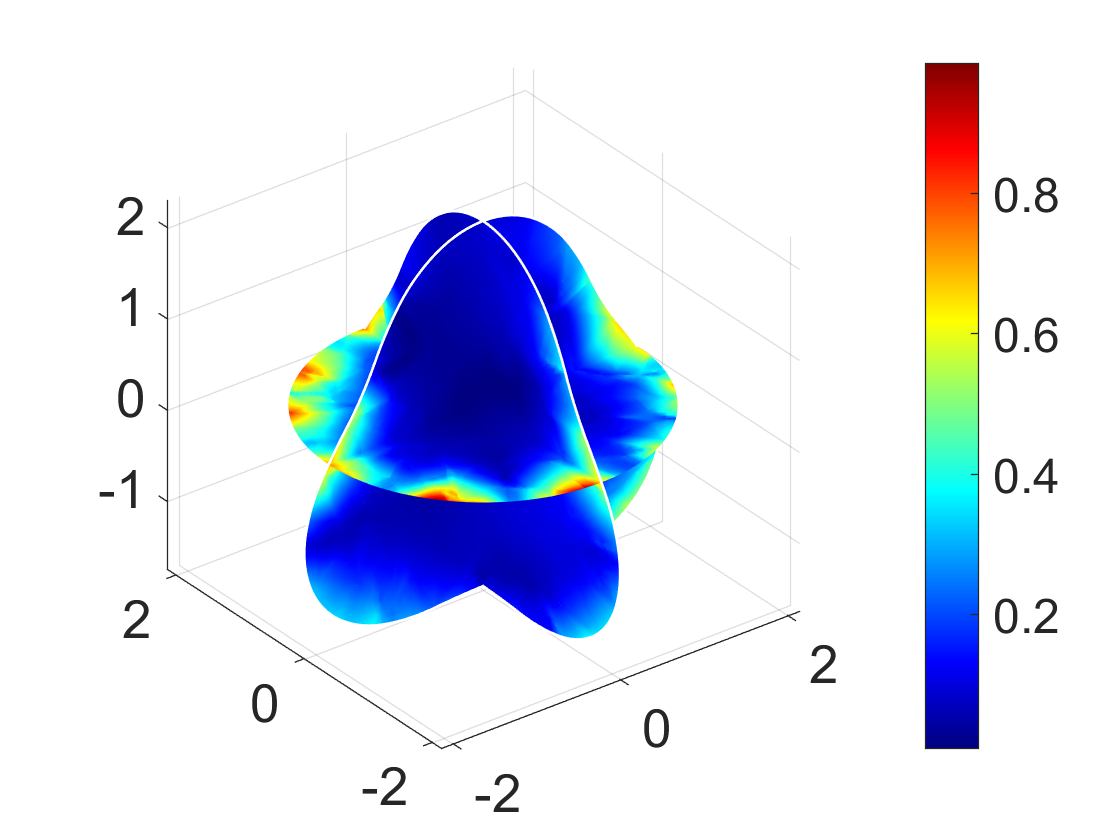}}\hfill
\caption{\label{fig:triangle-eigenfunction} Surface plots and multiple slices plots of the transmission eigenfunctions $|\bm E_{0,k}|$ with three different wavenumbers.  }

\end{figure}

\begin{figure}
\hfill\subfigure[isovalue $20$]{\includegraphics[width=0.33\textwidth]
                   {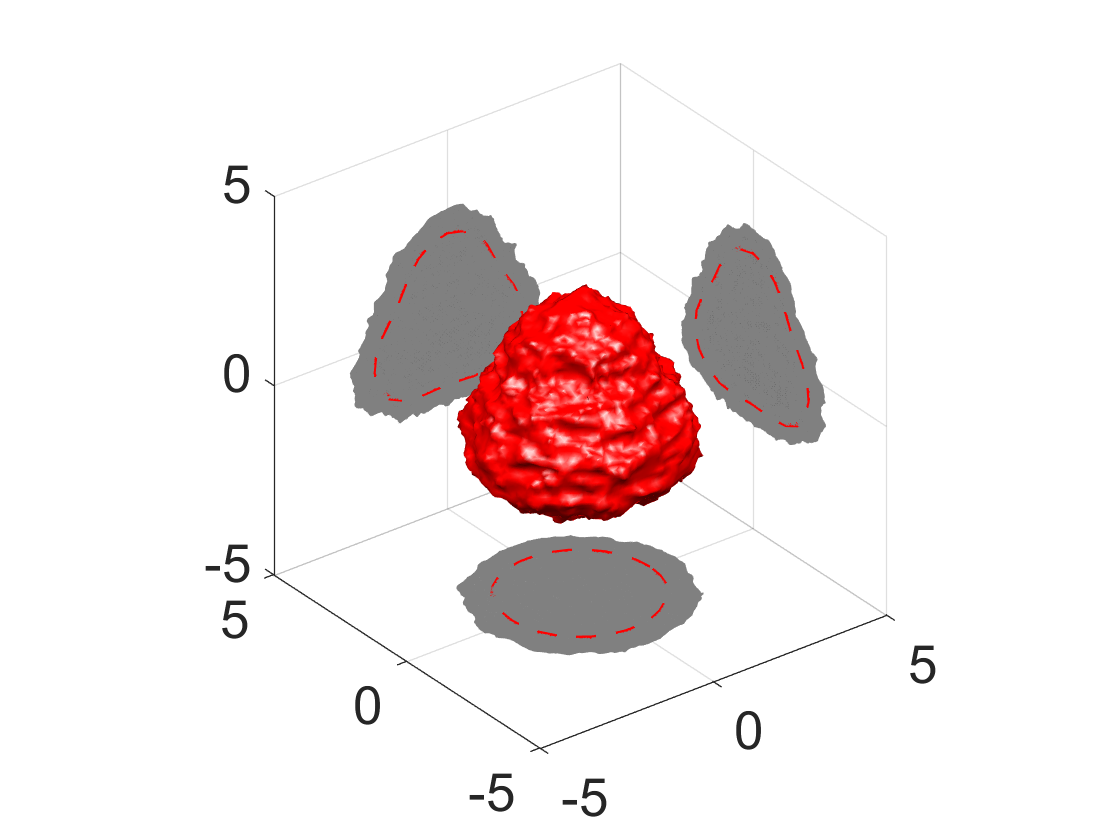}}\hfill
\hfill\subfigure[isovalue $30$]{\includegraphics[width=0.33\textwidth]
                   {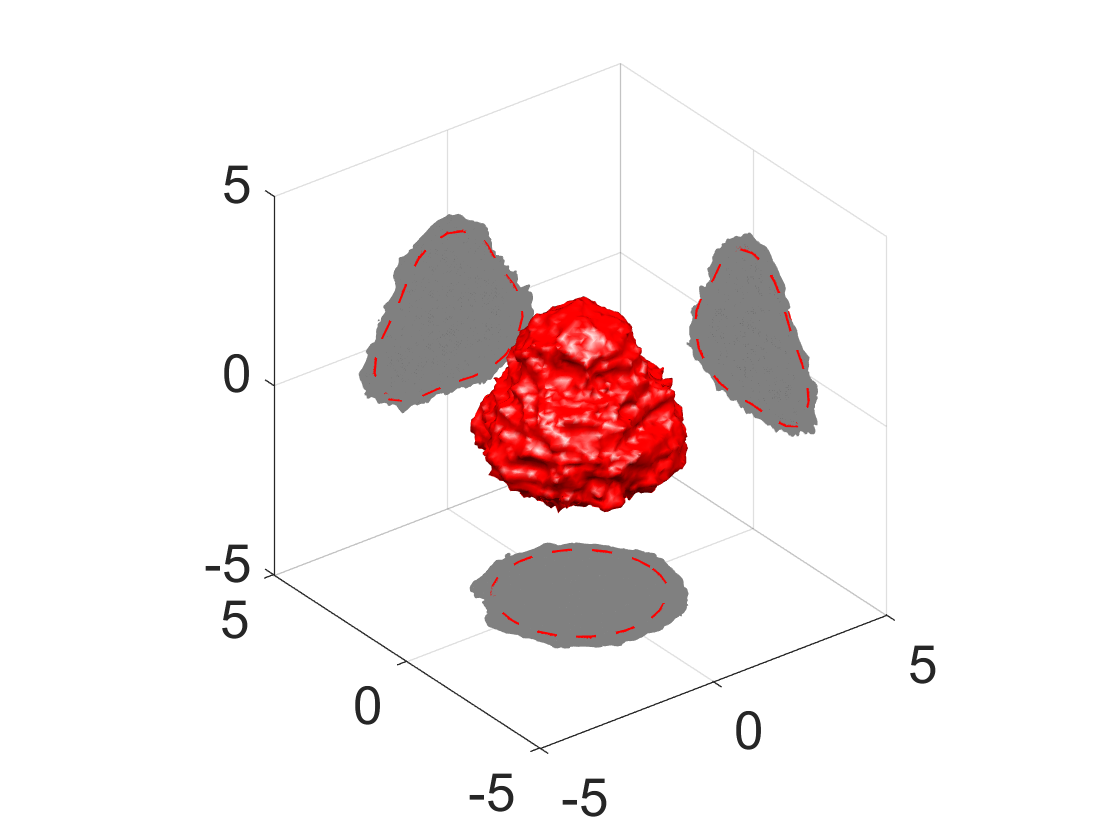}}\hfill
\hfill\subfigure[isovalue $40$]{\includegraphics[width=0.33\textwidth]
                   {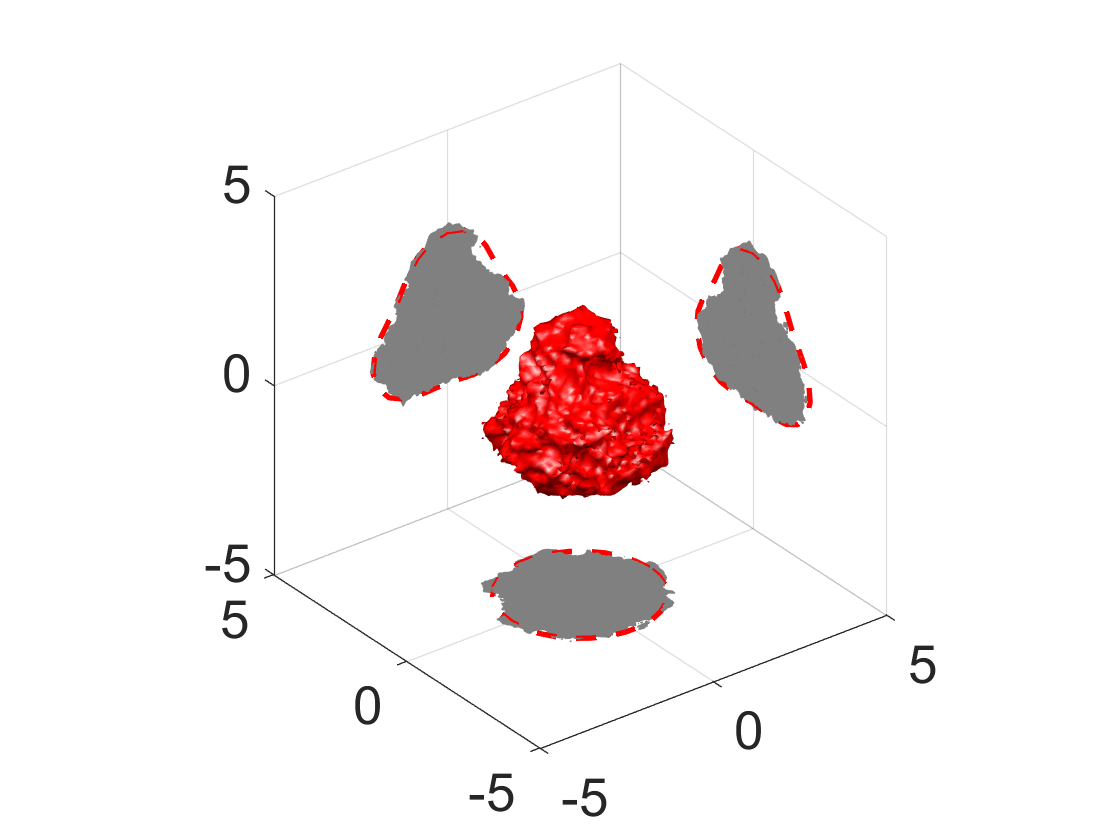}}\hfill\\
\hfill\subfigure[slice at $x=0$]{\includegraphics[width=0.33\textwidth]
                   {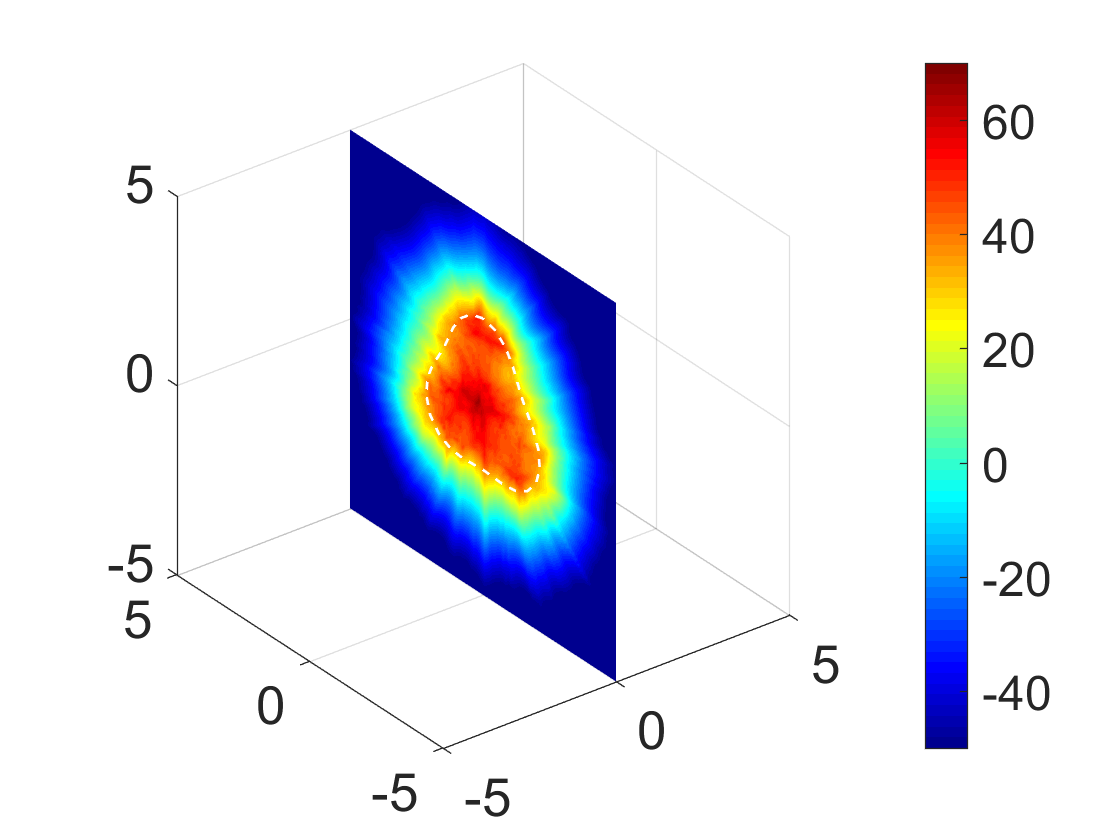}}\hfill
\hfill\subfigure[slice at $y=0$]{\includegraphics[width=0.33\textwidth]
                   {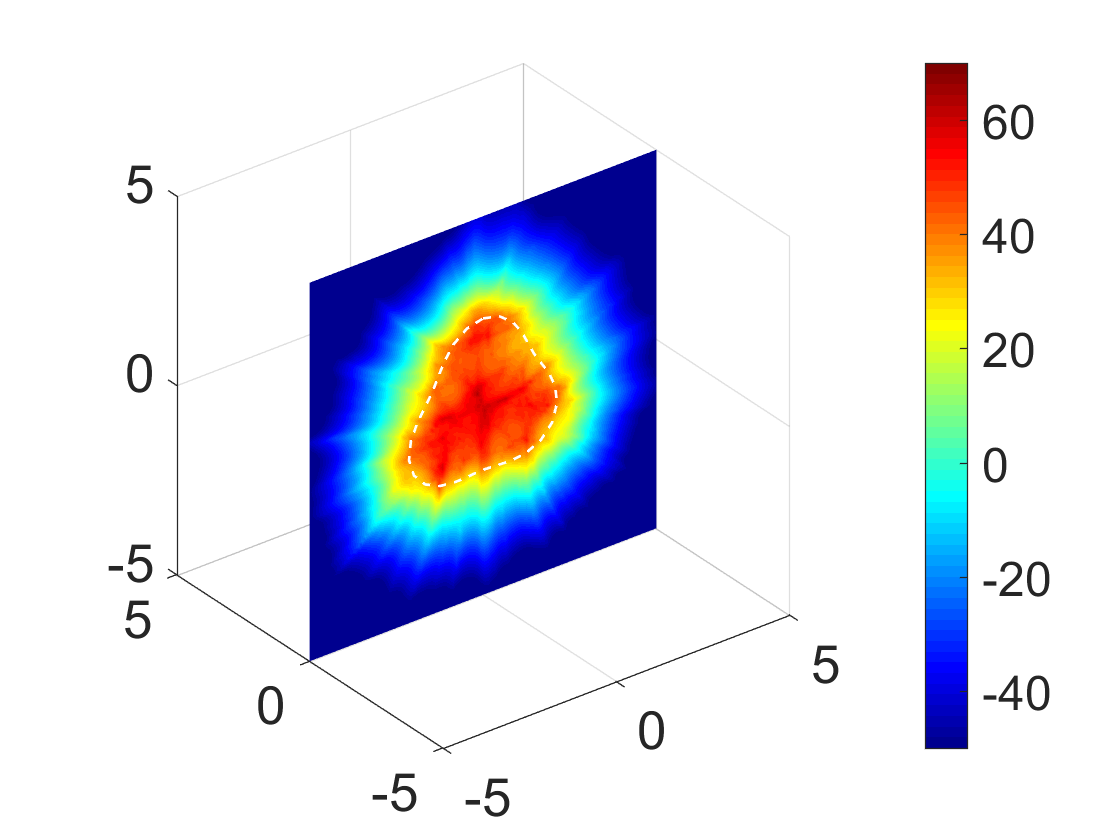}}\hfill
\hfill\subfigure[slice at $z=0$]{\includegraphics[width=0.33\textwidth]
                   {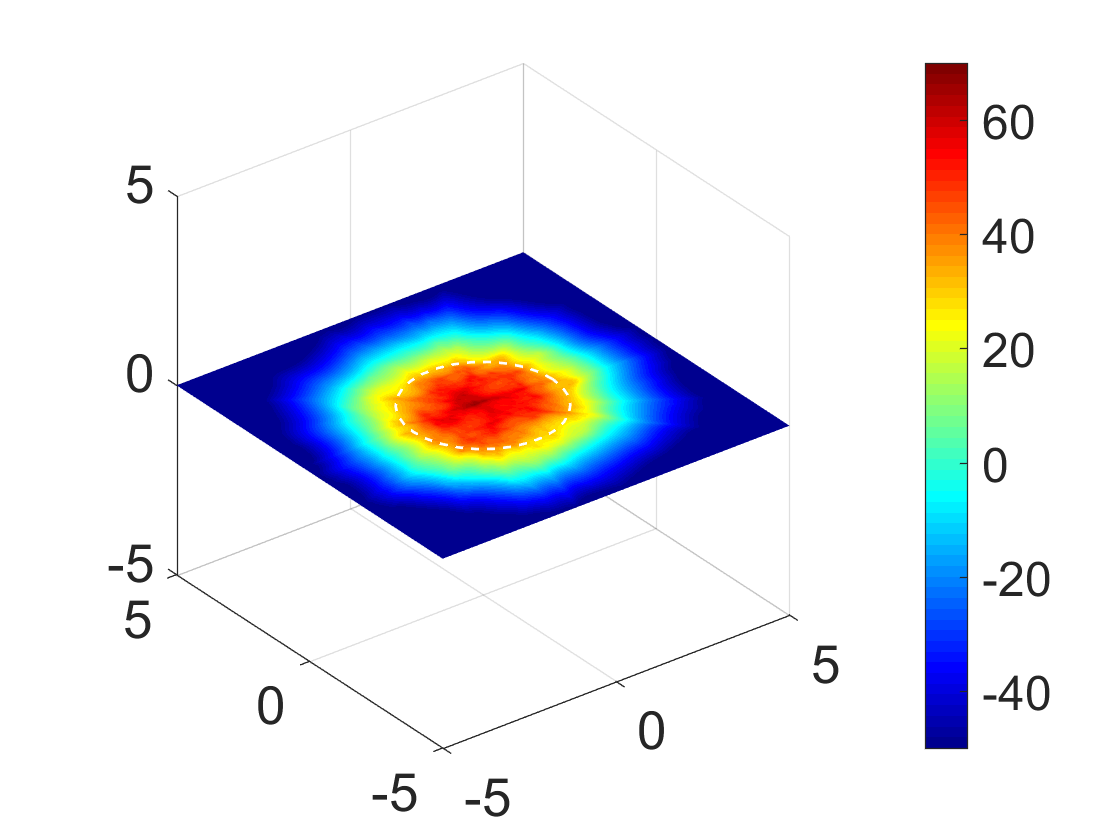}}\hfill
\caption{\label{fig:Gourd-MultiIndicator} Iso-surface plots and slice plots of multi-frequency indicator function $I_{\mathbb{K}_5}^{Res}$.  }

\end{figure}

In this example,  we consider the case that the scatterer coats a  thin layer of high refractive-index material.
It is worth mentioning that it requires heavy costs in computing the three-dimensional far-field data with high refractive index. Due to the limited computational resources, we could only afford to the two-dimensional calculation. Thus, we test the transverse electric (TE) mode in 2D.
The transmission eigenvalue problem \eqref{eq:elec_TEV_Pro} for TE mode is reduced to the following system:
\begin{equation*}\label{eq:TEreduce}
\begin{cases}
\underline{\nabla}\wedge(\nabla\wedge\mathbf{E})-k^2 \bm N \mathbf{E}=0,  \quad&\mbox{in}\ \ D,\medskip\\
\underline{\nabla}\wedge(\nabla\wedge\mathbf{E}_0)-k^2\mathbf{E}_0=0,   \quad&\mbox{in}\ \ D,\medskip\\
\bm \nu\wedge\mathbf{E}=\bm \nu\wedge\mathbf{E}_0, \ \underline{\bm \nu} \wedge (\nabla\wedge \mathbf{E})=\underline{\bm \nu} \wedge (\nabla\wedge\mathbf{E}_0) \quad & \mbox{on}\ \ \partial D,
\end{cases}
\end{equation*}
where
\begin{equation*}
\bm E:=\left(E^{(1)},\, E^{(2)}\right)^{\top}\quad   \text{and}\quad  \bm E_0:=\left(E_0^{(1)},\, E_0^{(2)}\right)^{\top},
\end{equation*}
and
\begin{equation*}
\nabla\wedge \mathbf{F}:=
 \partial_x  F^{(2)}-\partial_y   F^{(1)},\
 \underline{\nabla}\wedge  E:= \left(\begin{array}{c}
\partial_y  E \\
- \partial_x  E
\end{array}\right),\
\bm \nu \wedge \mathbf{F}:=
  \nu_1   F^{(2)}-\nu_2   F^{(1)},\
  \underline{\bm \nu}\wedge  E:= \left(\begin{array}{c}
\nu_2  E \\
- \nu_1  E
\end{array}\right).
\end{equation*}

Here we consider a kite-shaped domain with a thin layer, see Figure \ref{fig:Geometry}(b). The refractive indexes $\bm N$ of the outside layer and the inside domain are given by
\begin{equation*}
  \bm N_{out}(\bm x)=\left(
          \begin{array}{ccc}
            256 & 16  \\
          16& 256  \\
          \end{array}
        \right), \quad
  \bm N_{in}(\bm x)=\left(
          \begin{array}{ccc}
            2 & 0  \\
            0& 2  \\
          \end{array}
        \right).
\end{equation*}

Figure \ref{fig:kite_eigenfc} presents the exact transmission eigenfunctions in the predetermined domain. We can see from the refractive index $\bm N$ that this problem is associated with anisotropic media. In this case both transmission eigenfunctions $\bm E$ and $\bm E_0$ are localized around the boundary of the domain. This  example also numerically expands the theoretical result of isotropic cases \cite{DLWW}. The synthetic far-field data are computed within the interval $[1,2]$ with additional $1\%$ noise. First, we use Algorithm \uppercase\expandafter{\romannumeral1} to determine eight transmission eigenvalues. Next, we use the GTLS method with regularization parameter $\beta=10^{-6}$ to recover the Herglotz wave functions $E_{\bm g_0,k}$. We present the reconstruction results in Figure \ref{fig:kite_MultiIndicator} by using two, five and eight eigenmodes, respectively. One readily sees that the kite is already finely reconstructed with eight interior resonant modes. It is clear that the scale of the kite-shaped domain is much smaller than the underlying wavelength, $2\pi /k\approx 4$. This result shows that super-resolution reconstruction can be realized by the proposed imaging scheme. This is unobjectionably expected since we make use of the interior resonant modes for the reconstruction. Here, we want to point out that the contrast $\bm N$ just needed to be high around $\partial D$. Then small transmission eigenvalues can occur no matter the refractive index is large or normal inside the object. That means in real life, for a regular refractive inhomogeneity, one may first coat the object via indirect means with a relatively thin layer of high-contrast medium. For example, one could spray high-contrast material onto the surface of the scatterer. Then super-resolution imaging can be achieved by the same reconstruction procedure as above.

\begin{rem}
From Figure \ref{fig:Gourd-MultiIndicator} and \ref{fig:kite_MultiIndicator}, we can observe that the concave part of the scatterer can be reconstructed well. This is physically reasonable since the proposed method makes use of the interior resonant modes, which is equivalent to ``looking'' the scatterer from its inside. At this point, the concave part of the scatterer observed from the exterior becomes convex. 
\end{rem}

\subsubsection{TM mode in 2D}

\begin{figure}
\centering
\subfigure[$ E^{(1)}$]{\includegraphics[width=0.4\textwidth]
                   {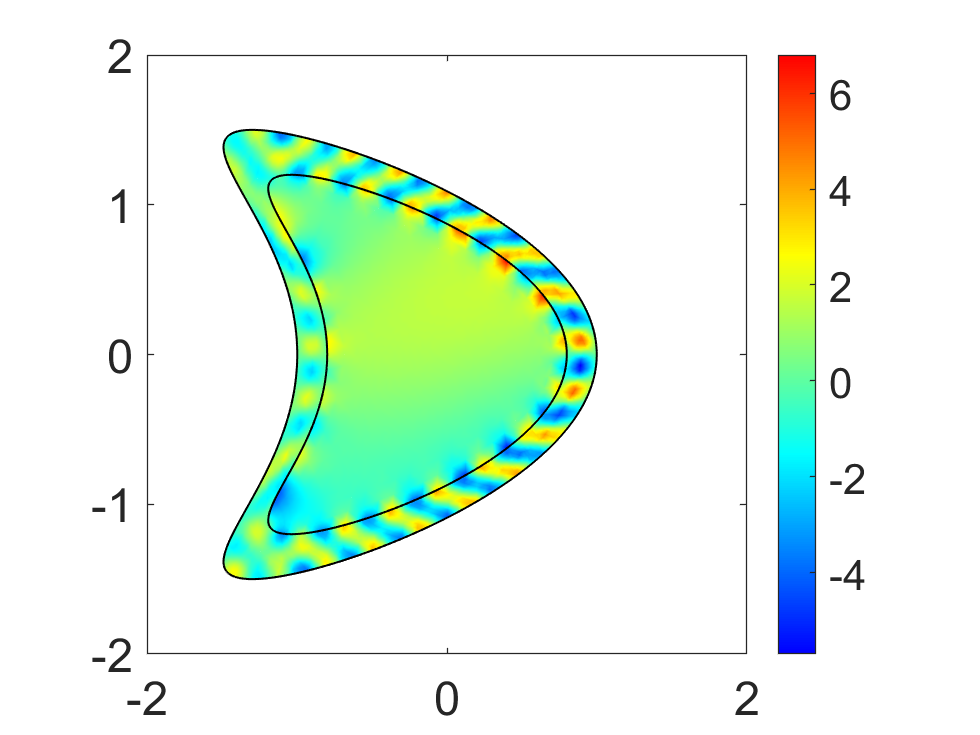}}
\subfigure[$ E^{(2)}$]{\includegraphics[width=0.4\textwidth]
                   {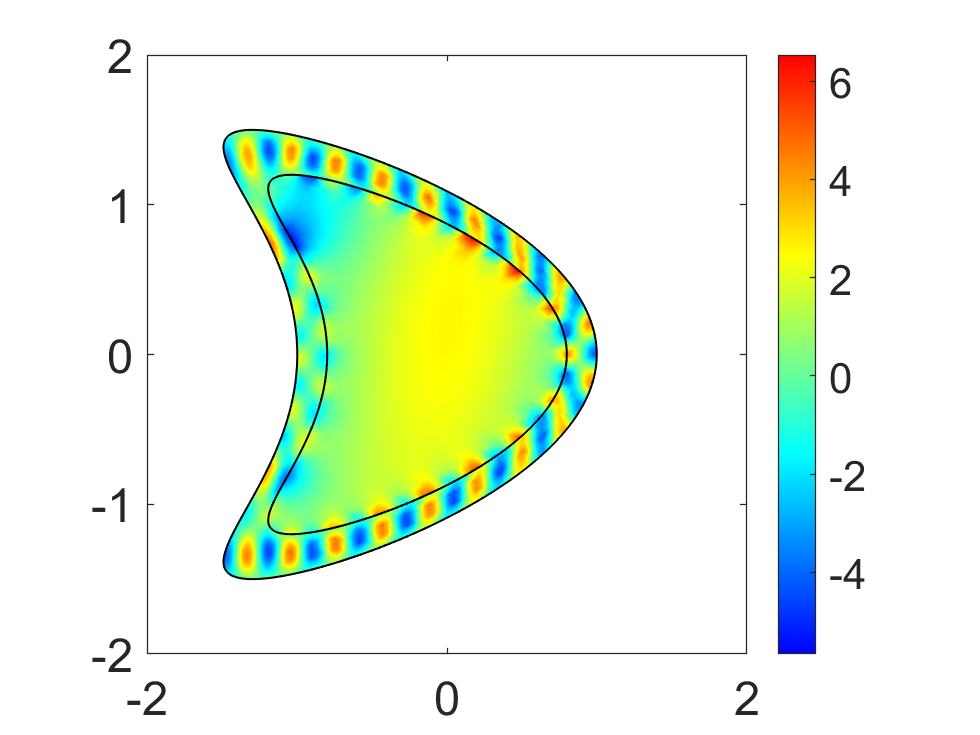}}\\
\subfigure[$ E_0^{(1)}$]{\includegraphics[width=0.4\textwidth]
                   {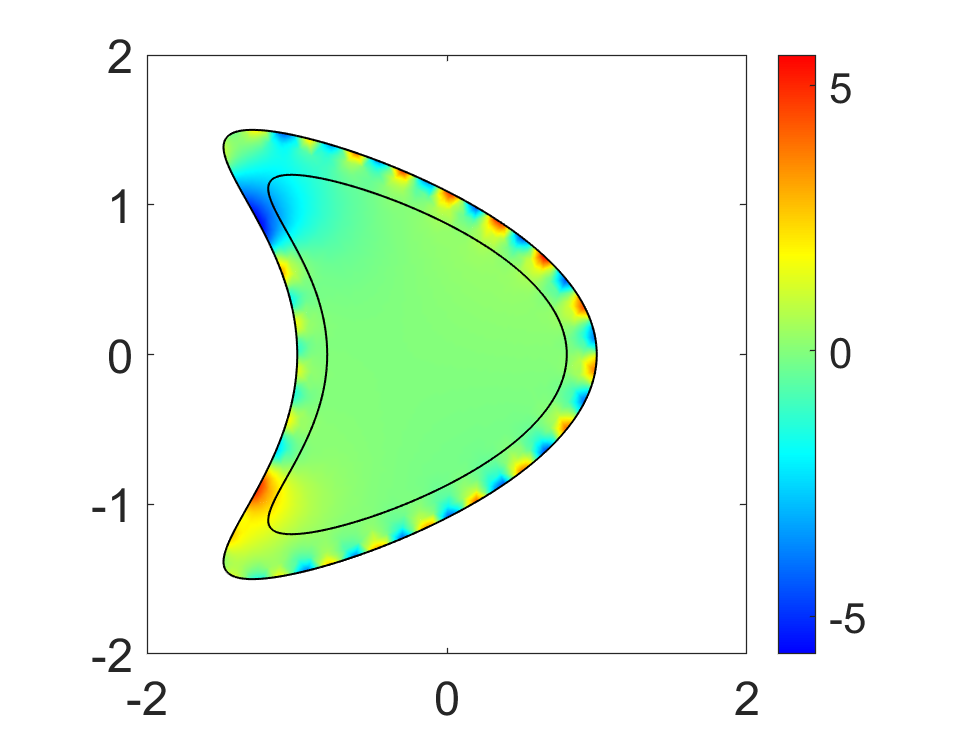}}
\subfigure[$ E_0^{(2)}$]{\includegraphics[width=0.4\textwidth]
                   {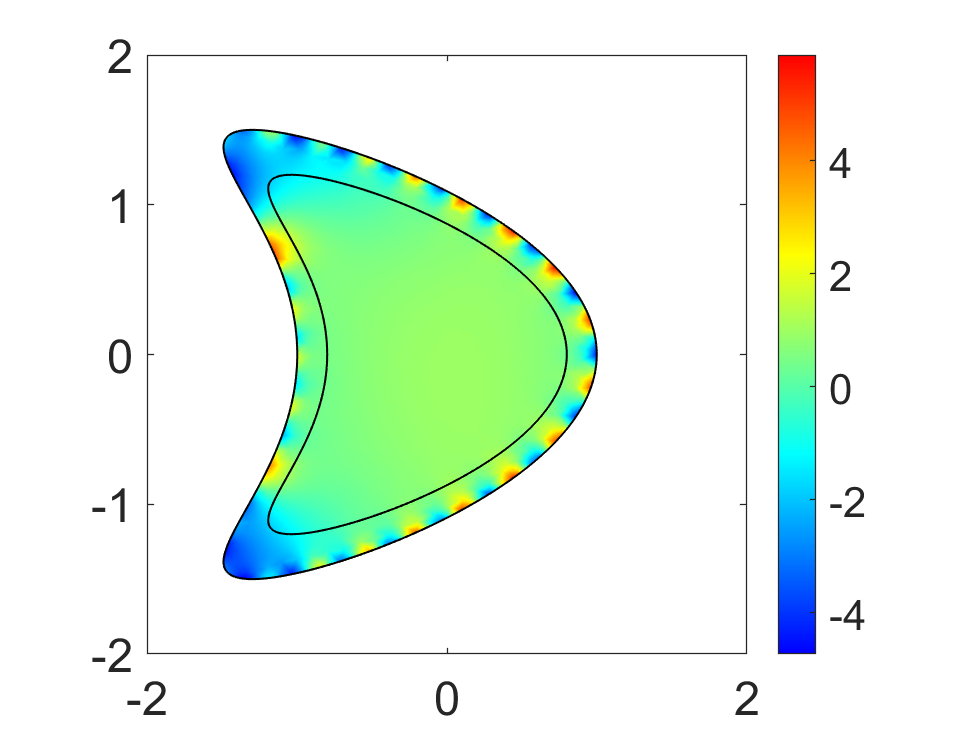}}
\caption{\label{fig:kite_eigenfc} Imagesc plots of the transmission eigenfunctions with $k=1.5098$. }

\end{figure}

\begin{figure}
\hfill\subfigure[$L=2$]{\includegraphics[width=0.33\textwidth]
                   {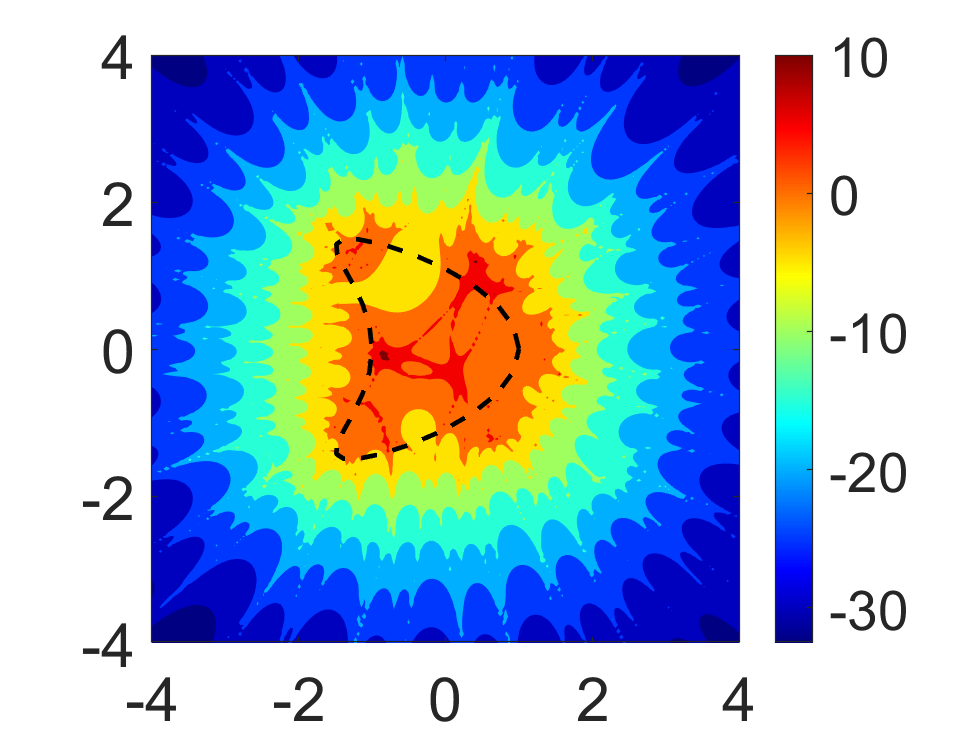}}\hfill
\hfill\subfigure[$L=5$]{\includegraphics[width=0.33\textwidth]
                   {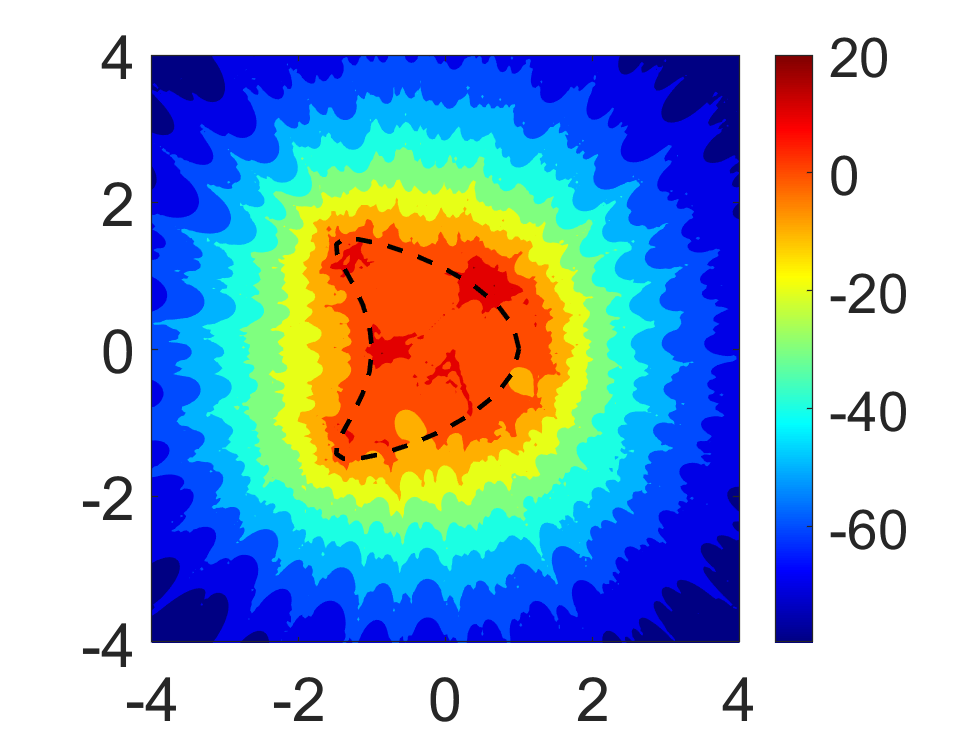}}\hfill
\hfill\subfigure[$L=8$]{\includegraphics[width=0.33\textwidth]
                   {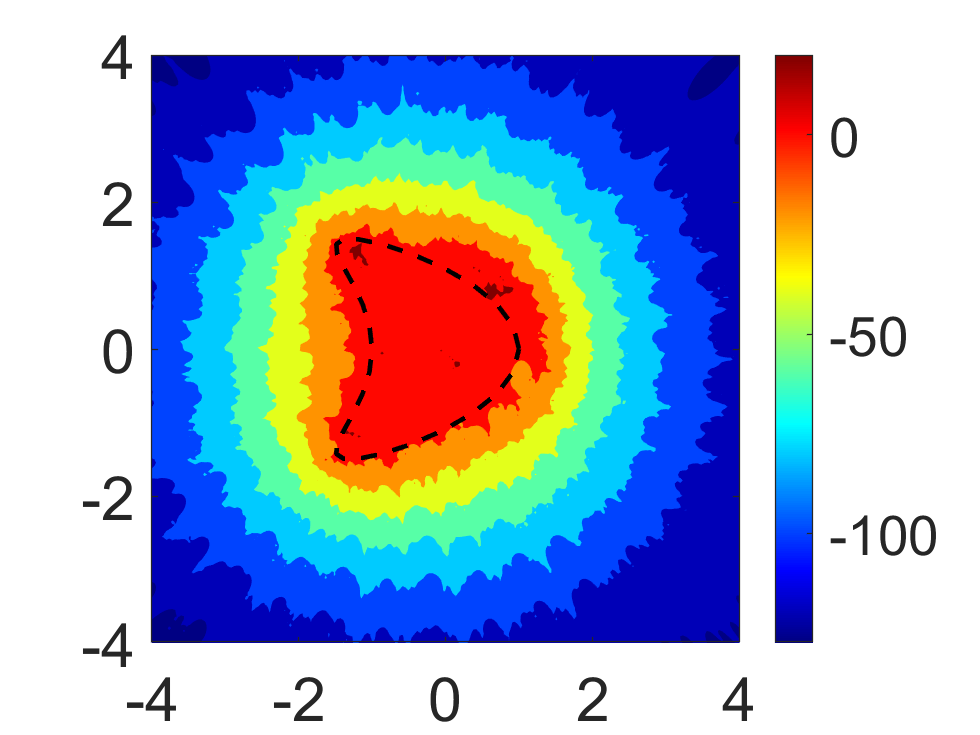}}\hfill\\
\caption{\label{fig:kite_MultiIndicator} Imagesc plots of multi-frequency indicator function $I_{\mathbb{K}_L}^{Res}$ with $L=2,\, 5, \, 8 $, respectively. }

\end{figure}

In the final example, we use the local geometric properties of transmission eigenfunctions to image the scatterer.  To reduce the computational cost,  we consider  the  transverse magnetic (TM) mode in 2D. The transmission eigenvalue problem \eqref{eq:elec_TEV_Pro} for  TM mode is reduced to the following system:

\begin{equation}\label{eq:TM2D}
\begin{cases}
\nabla\wedge(\underline{\nabla}\wedge E)-k^2 n(\bm x) E=0, \quad&\mbox{in}\ \ D,\medskip\\
\nabla\wedge(\underline{\nabla}\wedge E_0)-k^2 E_0=0, \quad&\mbox{in}\ \ D,\medskip\\
 \underline{\bm \nu} \wedge E= \underline{\bm \nu} \wedge E_0, \ \bm \nu\wedge (\underline{\nabla}\wedge E)=\bm \nu \wedge( \underline{\nabla}\wedge E_0)  \quad & \mbox{on}\ \ \partial D,
\end{cases}
\end{equation}
where the refractive index $n(\bm x)=1/4$. 
In particular, the system \eqref{eq:TM2D} can be rewritten as the following transmission eigenvalue problem associated with the scalar Helmholtz equation:
\begin{equation*}
\left\{
\begin{array}{ll}
\Delta E+k^2 n(\bm x)E=0  &\text{in} \ D,\medskip \\
\Delta E_0+k^2 E_0 =0 &\text{in} \ D, \medskip \\
\displaystyle{E=E_0,\ \ \frac{\partial E}{\partial \bm \nu}=\frac{\partial E_0}{\partial \bm \nu} } &\text{on} \ \partial D.
\end{array}
\right.
\end{equation*}
Here, we let $D$ be a square defined by $(x,y)\in [-1, 1]\times[-1, 1]$. The exact domain is shown in Figure \ref{fig:Geometry}(c).

\begin{figure}
\hfill\subfigure[$k_1=5.48$]{\includegraphics[width=0.33\textwidth]
                   {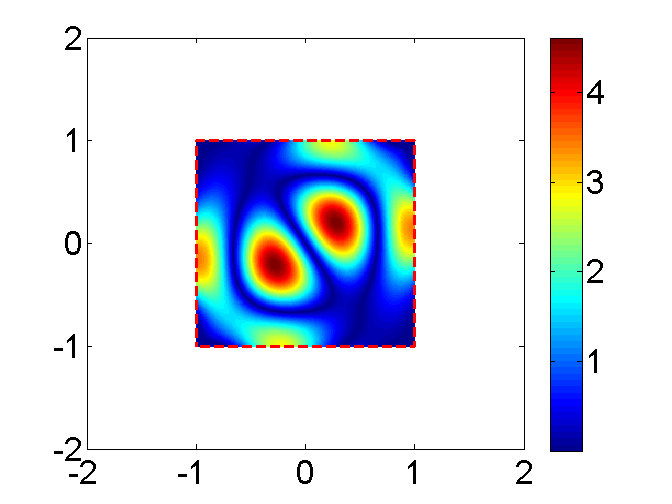}}\hfill
\hfill\subfigure[$k_2=6.10$]{\includegraphics[width=0.33\textwidth]
                   {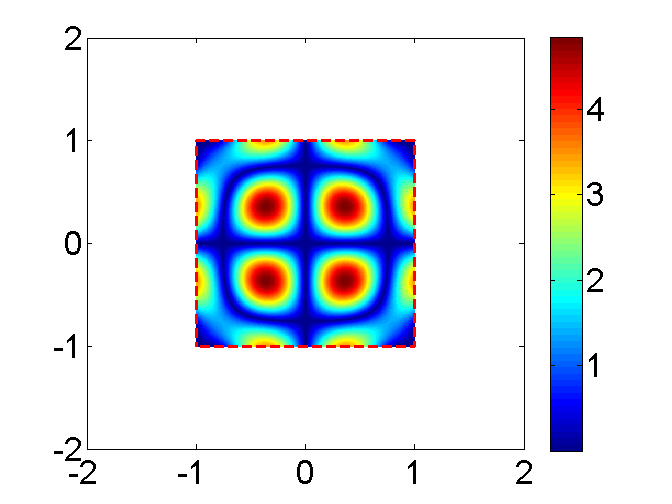}}\hfill
\hfill\subfigure[$k_3=6.65$]{\includegraphics[width=0.33\textwidth]
                   {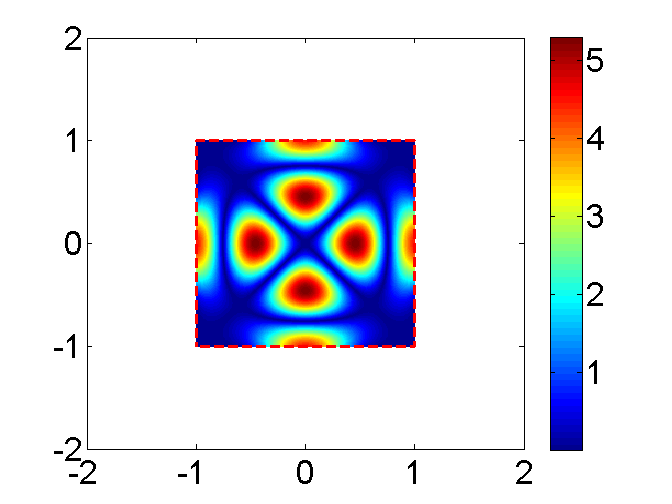}}\hfill\\
\hfill\subfigure[$k_1=5.40$]{\includegraphics[width=0.33\textwidth]
                   {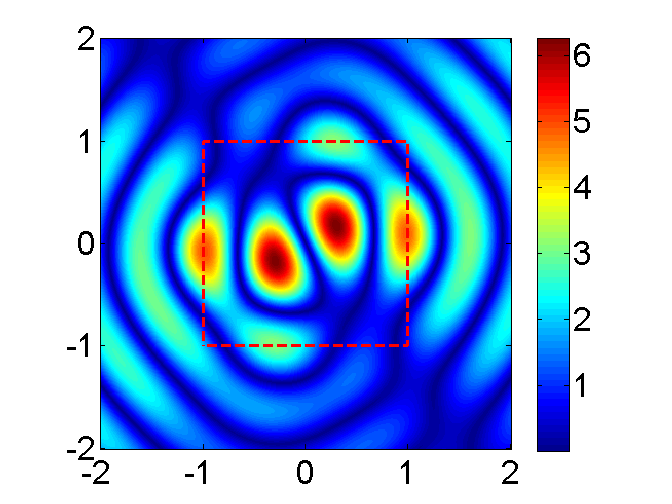}}\hfill
\hfill\subfigure[$k_2=6.11$]{\includegraphics[width=0.33\textwidth]
                   {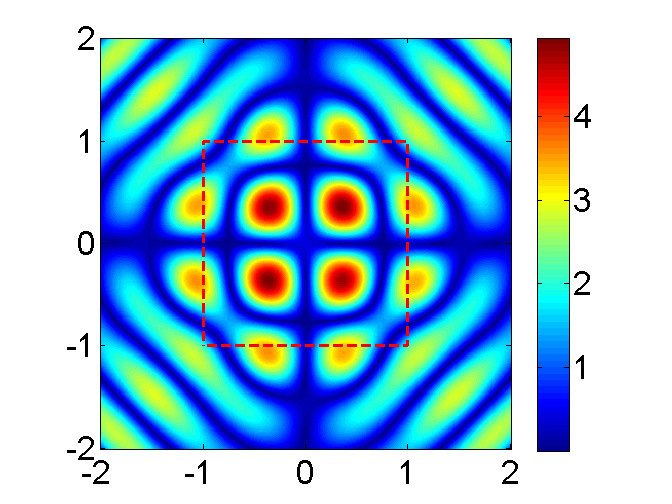}}\hfill
\hfill\subfigure[$k_3=6.67$]{\includegraphics[width=0.33\textwidth]
                   {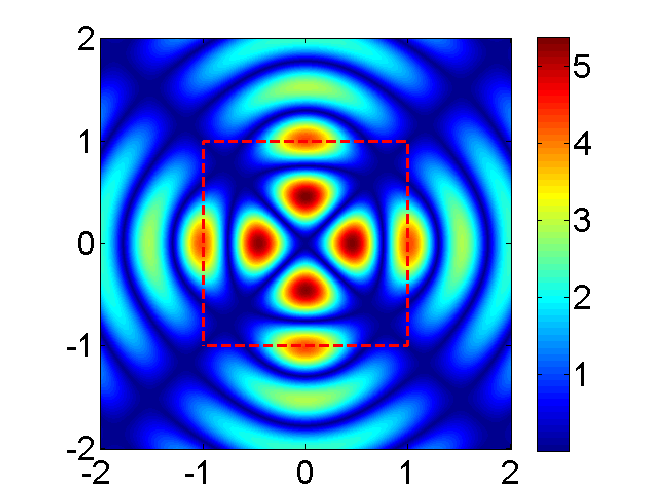}}\hfill\\
\hfill\subfigure[$k_1=5.40$]{\includegraphics[width=0.33\textwidth]
                   {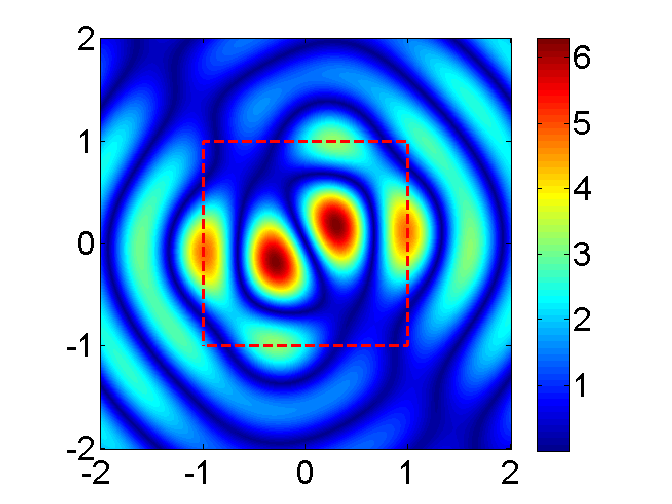}}\hfill
\hfill\subfigure[$k_2=6.11$]{\includegraphics[width=0.33\textwidth]
                   {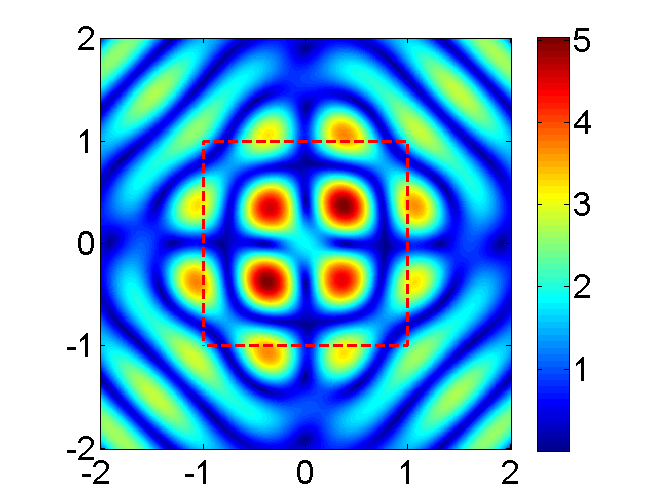}}\hfill
\hfill\subfigure[$k_3=6.67$]{\includegraphics[width=0.33\textwidth]
                   {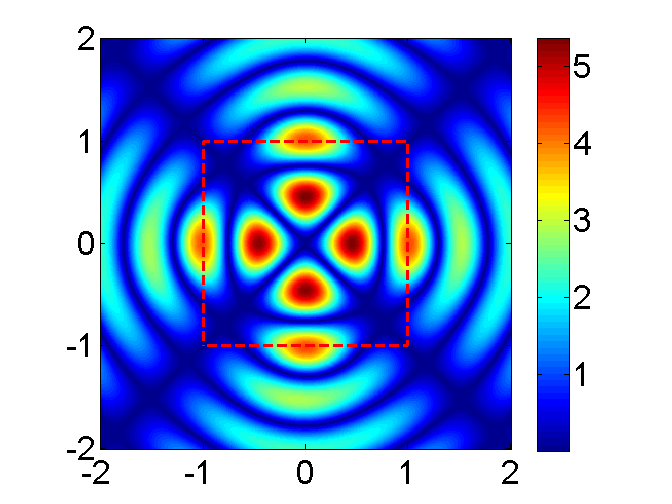}}\hfill
\caption{\label{fig:square_Vg} The top row: transmission eigenfunctions $E_{0,k}$ with the first three $k$; the middle row: the Herglotz wave recovered by FTLS method;
the bottom row: the Herglotz wave recovered by GTLS method.}
\end{figure}


\begin{figure}[h]
\hfill\subfigure[$L=2$]{\includegraphics[width=0.33\textwidth]
                   {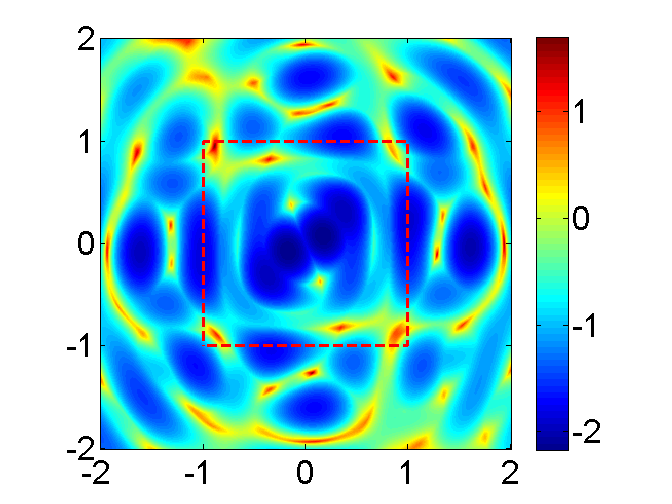}}\hfill
\hfill\subfigure[$L=4$]{\includegraphics[width=0.33\textwidth]
                   {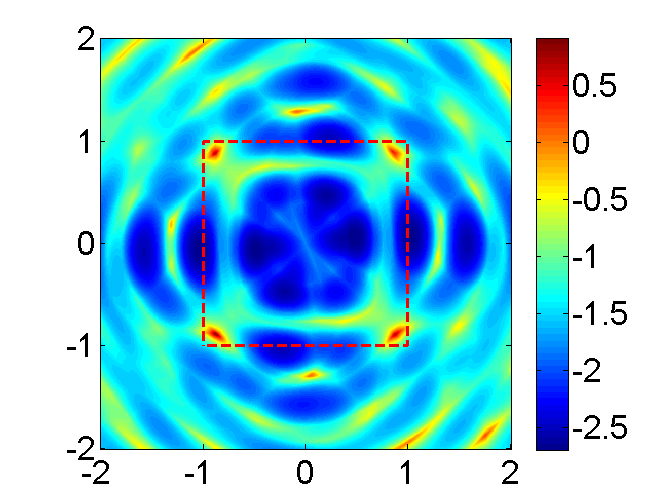}}\hfill
\hfill\subfigure[$L=6$]{\includegraphics[width=0.33\textwidth]
                   {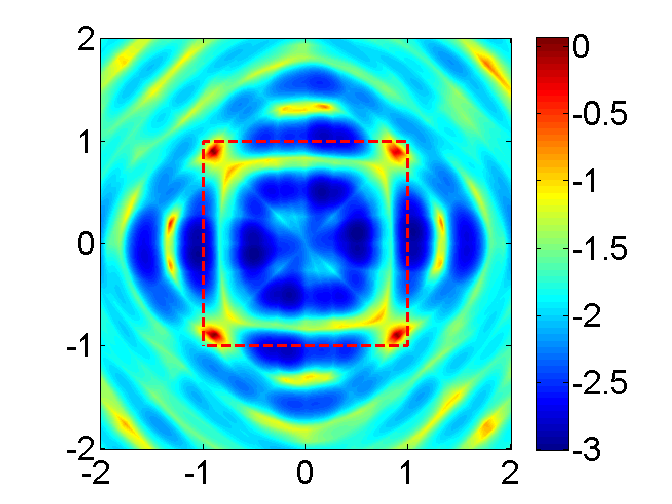}}\hfill\\
\hfill\subfigure[$L=2$]{\includegraphics[width=0.33\textwidth]
                   {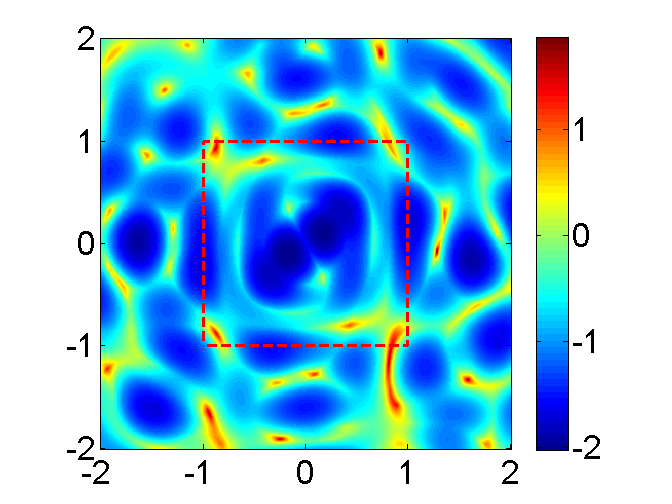}}\hfill
\hfill\subfigure[$L=4$]{\includegraphics[width=0.33\textwidth]
                   {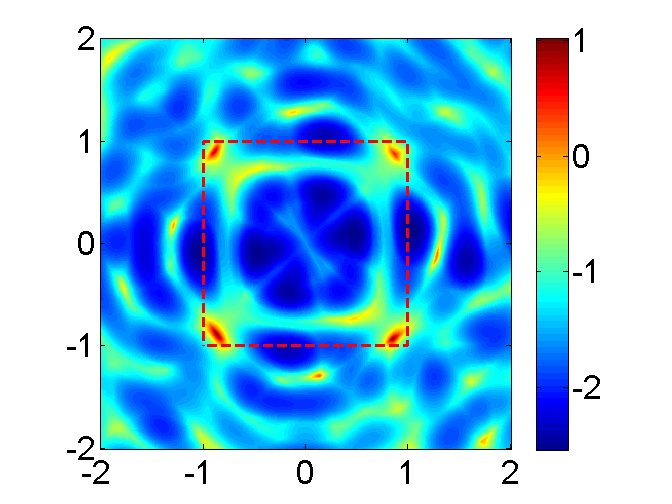}}\hfill
\hfill\subfigure[$L=6$]{\includegraphics[width=0.33\textwidth]
                   {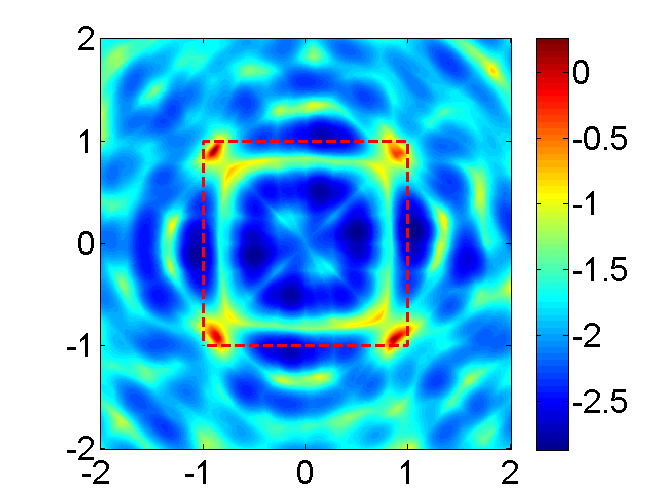}}\hfill\\
\caption{\label{fig:square_MultiIndicator} Imagesc plots of multi-frequency indicator function $I_{\mathbb{K}_L}^{Res}$ with $L=2,\, 4, \, 6 $, respectively. The top rows: recovered by the FTLS method ; the bottom rows: recovered by the GTLS method.  }

\end{figure}

To begin with, we proceed to determine  the transmission eigenfunctions $E_{0,k}$.  We consider the square domain as discussed above and in order to avoid inverse crimes, extra $5\%$ noise is added to the far-field data associated with those eigenvalues. For comparison, we use the FEM to compute the exact eigenfunction of $E_{0,k}$ from the exact domain, see the top row of Figure \ref{fig:square_Vg}.
In addition, the middle and bottom rows of Figure \ref{fig:square_Vg} present the Herglotz wave functions $E_{ {\bm g}_0,k}$ recovered by the FTLS and GTLS methods, respectively.  For the FTLS method, we take the Fourier truncation order by $N_t= 6$. For the GTLS method, the regularization parameter is chosen as $\beta=10^{-1}$. The dashed red lines denote the exact support of the square scatterer. It is clear that the reconstructed Herglotz wave functions by both methods are close to the exact transmission eigenfunctions $E_{0,k}$ inside the scatterer. This numerically verifies the conclusion in section 3. Moreover, from the top row of Figure \ref{fig:square_Vg}, one can find that the eigenfunctions are vanishing near the corners, which verifies the previous theoretical study \cite{BlastenA, Blasten17, BlastenLX}. In particular,  we can observe that the reconstructed eigenfunctions are also nearly vanishing near the corner points rather than at the corners. This is because that recovering the Herglotz kernel function  $\bm {g}_0$ from the far-field data is very ill-posed. 


Although the nodal lines of eigenfunctions appear in different locations for different eigenvalues (cf. \cite{CDLZ19, CDLZ20}), the nodal lines always go through the corners. So, the corners of the domain will stand out if we superimpose the indicator function with multi-frequency far-field data. Figure \ref{fig:square_MultiIndicator} presents the multi-frequency indicator function \eqref{eq:multi-indicator} with the first two, four and six eigenvalues with $5\%$ noise.  In this case, the truncated Fourier order is given by $N_t=8$ for the FTLS method and the regularization parameter is chosen as $\beta=10^{-3}$ for the GTLS method. One readily sees that the corner points are finely reconstructed with four interior resonant modes. If further a priori information is available on the shape, say, it is a polygon, then one can actually recover the scatterer. This example also shows that the proposed imaging scheme can break the Abbe resolution limit in recovering the fine details of $\partial D$, even the singular points.

\begin{rem}
It is worth to mention that the proposed method is valid for reconstructing an electromagnetic medium scatterer no matter if the medium is isotropic or anisotropic. In addition, the aforementioned  numerical experiments demonstrate that super-resolution reconstruction can be realized by the proposed imaging scheme.
\end{rem}


\noindent {\bf Acknowledgments.} 
The authors would like to thank Professor Wei Wu of Jilin University for providing many helpful discussions during the completion of this paper. The work was supported by the Hong Kong RGC General Research Fund (projects 12301420, 12302919, 11300821) and NSFC/RGC Joint Research Grant (project N\_CityU101/21).


\begin{thebibliography}{99}


\bibitem{Alberti} G. S. Alberti and Y. Capdeboscq, {\it Elliptic regularity theory applied to time harmonic anisotropic Maxwell's equations with less than Lipschitz complex coefficients}, SIAM Journal on Mathematical Analysis, {\bf 46} (2014), no. 1, 998--1016.

\bibitem{Ammari15} H. Ammari, J. Garnier, W. Jing, H. Kang, M. Lim, K. S$\varnothing$lna and H. Wang,  {\it Mathematical and Statistical Methods for Multistatic Imaging}, Springer, New York, 2013

\bibitem{BlastenA} E. Bl\aa sten, {\it Nonradiating sources and transmission eigenfunctions vanish at corners and edges}, SIAM J. Math. Anal., {\bf 50} (2018), no. 6, 6255--6270.

\bibitem{BlastenLLW} E. Bl\aa sten, X. Li, H. Liu and Y. Wang, {\it On vanishing and localization near cusps of transmission eigenfunctions: a numerical study}, Inverse Problems, {\bf 33} (2017), 105001.



\bibitem{Blasten16} E. Bl\aa sten and H. Liu, {\it On corners scattering stably and stable shape determination by a single far-field pattern}, Indiana Univ. Math. J., {\bf 70} (2021), 907--947.

\bibitem{Blasten17} E. Bl\aa sten and H. Liu, {\it On vanishing near corners of transmission eigenfunctions}, J. Funct. Anal., {\bf 273} (2017), 3616--3632. Addendum available from https://arxiv.org/abs/1710.08089.

\bibitem{Blasten21} E. Bl\aa sten and H. Liu, {\it Scattering by curvatures, radiationless sources, transmission eigenfunctions, and inverse scattering problems}, SIAM J. Math. Anal., {\bf 53} (2021), 3801--3837.

\bibitem{BlastenLX} E. Bl\aa sten, H. Liu and J. Xiao, {\it On an electromagnetic problem in a corner and its applications}, Anal. PDE, \textbf{14} (2021), no. 7, 2207--2224.

\bibitem{CC06} F. Cakoni and D. Colton, {\it Qualitative Methods in Inverse Scattering Theory}, Springer, Berlin, 2006.

\bibitem{CCH16} F. Cakoni, D. Colton and H. Haddar, {\it Inverse Scattering Theory and Transmission Eigenvalues}, SIAM, Philadelphia, 2016.

\bibitem{CCM11} F. Cakoni, D. Colton and P. Monk, {\it The Linear Sampling Method in Inverse Electromagnetic Scattering}, SIAM, Philadelphia, 2011.

%

\bibitem{CDLZ19} X. Cao, H. Diao, H. Liu and J. Zou, {\it On novel geometric structures of Laplacian eigenfunctions in $\mathbb{R}^3$ and applications to inverse problems}, SIAM J. Math. Anal., {\bf 53} (2021), 1263--1294.

\bibitem{CDLZ20} X. Cao, H. Diao, H. Liu and J. Zou, {\it On nodal and generalized singular structures of Laplacian eigenfunctions and applications to inverse scattering problems}, J. Math. Pures Appl., {\bf 143} (2020), 116--161.

\bibitem{CDHLW} Y.-T. Chow, Y. Deng, Y. He, H. Liu and X. Wang, {\it Surface-localized transmission eigenstates, super-resolution imaging and pseudo surface plasmon modes}, SIAM J. Imaging Sci., {\bf 14} (2021), 946--975.

\bibitem{CDLS} Y.-T. Chow, Y. Deng, H. Liu and M. Sunkula, {\it Surface concentration of transmission eigenfunctions}, arXiv:2109.14361.

\bibitem{ColtonKress19} D. Colton and R. Kress, {\it Inverse Acoustic and Electromagnetic Scattering Theory}, 4th Edition, Springer, New York, 2019.

\bibitem{ColtonKress01} D. Colton and R. Kress, {\it On the denseness of Herglotz wave functions and electromagnetic Herglotz pairs in Sobolev spaces}, Math. Meth. Appl. Sci, {\bf 24} (2001), no. 16, 1289--1303.

\bibitem{DDL} Y. Deng, C. Duan and H. Liu, {\it On vanishing near corners of conductive transmission eigenfunctions}, Res. Math. Sci., \textbf{9} (2022), 2.

\bibitem{DJLZ} Y. Deng, Y. Jiang, H. Liu and K. Zhang, {\it On new surface-localized transmission eigenmodes}, Inverse Problems and Imaging, (2021), https://doi.org/10.3934/ipi.2021063

\bibitem{DLWW} Y. Deng, H. Liu, X. Wang, and W. Wu, {\it On geometric properties of electromagnetic transmission eigenfunctions and artificial mirage}, SIAM J. Appl. Math., (2021), https://doi.org/10.1137/21M1413547

\bibitem{DCL} H. Diao, X. Cao and H. Liu, {\it On the geometric structures of transmission eigenfunctions with a conductive boundary condition and applications}, Comm. Partial Differential Equations, {\bf 46} (2021), 630--679.

\bibitem{DLWY} H. Diao, H. Liu, X. Wang and K. Yang, {\it On vanishing and localizing around corners of electromagnetic transmission resonances}, Partial Differential Equations and Applications, {\bf 2} (2021), 1--20.







\bibitem{KM95} A. Kirsch and P. Monk, {\it A finite element/spectral method for approximating the time-harmonic Maxwell system in $\mathbb{R}^3$}, SIAM J. Appl. Math., {\bf 55} (1995), 1324--1344. 

\bibitem{Liu20} H. Liu, {\it On local and global structures of transmission eigenfunctions and beyond}, J. Inverse Ill-Posed Probl., (2020), https://doi.org/10.1515/jiip-2020-0099.

\bibitem{LLWW} H. Liu, X. Liu, X. Wang and Y. Wang, {\it On a novel inverse scattering scheme using resonant modes with enhanced imaging resolution}, Inverse Problems, {\bf 35} (2019), 125012.

\bibitem{LWZ} H. Liu, Y. Wang, and S. Zhong, {\it Nearly non-scattering electromagnetic wave set and its application}, Z. Angew. Math. Phys, {\bf 68} (2017), no. 2, 35.

\bibitem{Marschall} J. Marschall, {\it The trace of Sobolev-Slobodeckij spaces on Lipschitz domains}, Manuscripta Math., {\bf 58} (1987), no. 1, 47--65.

\bibitem{Monk} P. Monk, {\it Finite Element Methods for Maxwell's Equations}, Oxford University Press, New York, 2003.

\bibitem{Monk2012} P. Monk and J. Sun, {\it Finite element methods for Maxwell's transmission eigenvalues}, SIAM J. Sci. Comput., \textbf{34} (2012), no. 3, B247--B264.





\end{thebibliography}
\end{document}